\documentclass[11pt,reqno]{amsart}

\usepackage{floatrow}

\usepackage{amsmath}
\usepackage{amssymb}
\usepackage[left=1in,top=1in,right=1in,bottom=1in]{geometry}
\usepackage{epsfig}
\usepackage[normalem]{ulem}
\usepackage[usenames,dvipsnames]{color}
\usepackage{todonotes}
\usepackage[colorlinks, citecolor=black, linkcolor=black, urlcolor=black]{hyperref}
\usepackage{bm}
\usepackage{algorithm}
\usepackage{algpseudocode}

\usepackage{enumitem}

\allowdisplaybreaks

\usepackage{hyperref}
\usepackage{cleveref}
\crefname{subsection}{subsection}{subsections}

\usepackage{tikz}
\usetikzlibrary{backgrounds}
\usetikzlibrary{intersections}

\newtheorem{theorem}{Theorem}[section]
\newtheorem{lemma}[theorem]{Lemma}

\newtheorem{observation}[theorem]{Observation}
\newtheorem{algorithms}[theorem]{Algorithm}
\newtheorem{corollary}[theorem]{Corollary}

\newcommand\eps{\varepsilon}
\newcommand{\E}{\mathbb E}

\newcommand{\Prob}{\mathbb{P}}

\newcommand{\N}{{\mathbb N}}

\newcommand{\scr}{\mathcal}
\newcommand{\mb}{\mathbb}

\theoremstyle{definition}

\newcommand{\bigo}{\mathcal{O}}

\title{Cliques, Chromatic Number, and Independent Sets in the Semi-random Process}

\author{David Gamarnik}
\address{MIT Sloan School of Management, MIT, Cambridge, USA}
\email{gamarnik@mit.edu}

\author{Mihyun Kang}
\address{Institute of Discrete Mathematics, Graz University of Technology, Graz, Austria}
\email{kang@math.tugraz.at}

\author{Pawe\l{} Pra\l{}at}
\address{Department of Mathematics, Toronto Metropolitan University, Toronto, Canada}
\email{pralat@torontomu.ca}

\date{}

\begin{document}

\maketitle

\begin{abstract}
The semi-random graph process is a single player game in which the player is initially presented an empty graph on $n$ vertices. In each round, a vertex $u$ is presented to the player independently and uniformly at random. The player then adaptively selects a vertex $v$, and adds the edge $uv$ to the graph. For a fixed monotone graph property, the objective of the player is to force the graph to satisfy this property with high probability in as few rounds as possible. In this paper, we investigate the following three properties: containing a complete graph of order $k$, having the chromatic number at least $k$, and not having an independent set of size at least $k$.
\end{abstract}

\section{Introduction and Main Results}

\subsection{Definitions} 

In this paper, we consider the \textbf{semi-random graph process} suggested by Peleg Michaeli, introduced formally in~\cite{beneliezer2019semirandom}, and studied recently in~\cite{behague2022,beneliezer2020fast,burova2022semi,frieze2022hamilton,gao2020hamilton,gao2022fully,gao2022perfect,macrury2022sharp,frieze2023building,molloy2023matchings} that can be viewed as a ``one player game''. The process starts from $G_0$, the empty graph on the vertex set $[n]:=\{1,\ldots,n\}$ where $n\in \mathbb N$. In each \textbf{round} $t\in \mathbb N$, a vertex $u_t$ is chosen uniformly at random from $[n]$. Then, the player (who is aware of graph $G_t$ and vertex $u_t$) must select a vertex $v_t$ and add the edge $u_tv_t$ to $G_t$ to form $G_{t+1}$. The goal of the player is to build a (multi)graph satisfying a given property $\scr{P}$ as quickly as possible. It is convenient to refer to $u_t$ as a {\bf square}, and $v_t$ as a {\bf circle} so every edge in $G_t$ joins a square with a circle. We say that $v_t$ is paired to $u_t$ in step $t$. Moreover, we say that vertex $x \in [n]$ is \textbf{covered} by the square $u_t$ arriving at round $t$, provided $u_t = x$. The analogous definition extends to the circle $v_t$. Equivalently, we may view $G_t$ as a directed graph where each arc  directs from $u_t$ to $v_t$, and thus we may use $(u_t,v_t)$ to denote the edge added in step~$t$. For this paper, it is easier to consider squares and circles for counting arguments.

\medskip

A \textbf{strategy} $\scr{S}$ is defined by specifying for each $n \ge 1$, a sequence of functions $(f_{t})_{t=1}^{\infty}$, where for each $t \in \mb{N}$, $f_t(u_1,v_1,\ldots, u_{t-1},v_{t-1},u_t)$ is a distribution over $[n]$ which depends on the vertex $u_t$, and the history of the process up until step $t-1$. Then, $v_t$ is chosen according to this distribution. If $f_t$ is an atomic distribution, then $v_t$ is determined by $u_1,v_1, \ldots ,u_{t-1},v_{t-1},u_t$. We then denote $(G_{i}^{\scr{S}}(n))_{i=0}^{t}$ as the sequence of random (multi)graphs obtained by following the strategy $\scr{S}$ for $t$ rounds; where we shorten $G_{t}^{\scr{S}}(n)$ to $G_t$ or $G_{t}(n)$ when clear. 

\subsection{Notation}

Results presented in this paper are asymptotic by nature. We say that some property $\mathcal P$ holds \textbf{asymptotically almost surely} (or \textbf{a.a.s.}) if the probability that the semi-random process has this property (after possibly applying some given strategy) tends to $1$ as $n$ goes to infinity. 
Given two functions $f=f(n)$ and $g=g(n)$, we will write $f(n)=\bigo(g(n))$ if there exists an absolute constant $c \in \mathbb R_+$ such that $|f(n)| \leq c|g(n)|$ for all $n$, $f(n)=\Omega(g(n))$ if $g(n)=\bigo(f(n))$, $f(n)=\Theta(g(n))$ if $f(n)=\bigo(g(n))$ and $f(n)=\Omega(g(n))$, and we write $f(n)=o(g(n))$ or $f(n) \ll g(n)$ if $\lim_{n\to\infty} f(n)/g(n)=0$. In addition, we write $f(n) \gg g(n)$ if $g(n)=o(f(n))$ and we write $f(n) \sim g(n)$ if $f(n)=(1+o(1))g(n)$, that is, $\lim_{n\to\infty} f(n)/g(n)=1$.

\medskip

We will use $\log n$ to denote a natural logarithm of $n$. As mentioned earlier, for a given $n \in \N := \{1, 2, \ldots \}$, we will use $[n]$ to denote the set consisting of the first $n$ natural numbers, that is, $[n] := \{1, 2, \ldots, n\}$. Finally, as typical in the field of random graphs, for expressions that clearly have to be an integer, we round up or down but do not specify which: the choice of which does not affect the argument.

\subsection{Main Results---Complete Graphs} 

In this paper, we investigate three monotone properties. The first one is the property of containing $K_k$, a complete graph of order $k$. In the very first paper on the semi-random process~\cite{beneliezer2019semirandom}, it was proved that a.a.s.\ one may construct a complete graph of a constant order $k$ once there are vertices with at least $k$ squares on them. On the other hand, if no vertex receives at least $k$ squares, it is impossible to achieve it a.a.s. Specifically, the following result was proved.

\begin{observation}[\cite{beneliezer2019semirandom}]\label{obs:cliques1}
Fix an integer $k \ge 3$ and any function $\omega=\omega(n)$ that tends to infinity as $n \to \infty$. Then, the following hold.
\begin{itemize}
\item[(a)] There exists a strategy that a.a.s.\ creates $K_k$ at time $t = \omega n^{(k-2)/(k-1)}$. 
\item[(b)] There is no strategy that a.a.s.\ creates $K_k$ at time $t = n^{(k-2)/(k-1)} / \omega$. 
\end{itemize}
\end{observation}

In fact, part~(a) of the above observation was proved for a larger family of graphs that are $k-1 \ (\, \ge 2)$ degenerate (see Section~\ref{sec:chi} for the definition of degeneracy). Moreover, it was conjectured that part~(b) can be generalized to such large family of graphs. The conjecture was proved recently in~\cite{behague2022}. As a result, creating graphs of a constant size is well-understood---essentially, creating a fixed graph with degeneracy $d$ is possible once the process lasts long enough so that there are vertices with at least $d-1$ squares. 

\medskip

On the other hand, constructing complete graphs of order $k \gg \log n$ is very simple and can be done in almost optimal way. It follows immediately from Chernoff's bound (see~(\ref{chern1}) and~(\ref{chern})), together with the union bound over all vertices, that if $t \gg n \log n$, then a.a.s.\ all vertices receive 
$$
\frac {t}{n} \left(1+ \bigo \left( \sqrt{ \frac {\log n}{t/n} } \right) \right) \sim \frac {t}{n}
$$
squares. One may try to create a complete graph on the vertex set $[k]$ for $k=2\ell+1$ ($\ell \in \N$) by connecting the $j$th square ($j \in [\ell]$) landing on vertex $i$ with vertex $(i-1+j) \pmod{2\ell+1} + 1$. This simple algorithm yields a lower bound for the size of the complete graph. To get an upper bound, we simply observe that it is impossible to create $K_{k'}$ if the maximum degree is smaller than $(k'-1)/2$. After combining the two observations, we get the following. 

\begin{observation}\label{obs:cliques2}
Suppose that $t = t(n) \ge \omega n \log n$, where $\omega=\omega(n)$ is any function that tends to infinity as $n \to \infty$. Let 
\begin{eqnarray*}
k &=& k(n) \ := \  \frac {2t}{n} \left( 1-\omega^{-1/3} \right) \ \sim\ \frac {2t}{n}\\ 
k' &=&  k'(n)\ :=\  \frac {2t}{n} \left( 1+\omega^{-1/3} \right) \ \sim\ \frac {2t}{n}  \ \sim\  k.
\end{eqnarray*} 
Then, the following hold.
\begin{itemize}
\item[(a)] There exists a strategy that a.a.s.\ creates $K_{\min\{k,n\}}$ at time $t$. 
\item[(b)] There is no strategy that a.a.s.\ creates $K_{k'}$ at time $t$. 
\end{itemize}
\end{observation}

In fact, much stronger property holds. Let $H$ be an $n$-vertex graph of maximum degree $\Delta = \Delta(n) \gg \log n$. In~\cite{beneliezer2020fast} it was proved that there exists a strategy to build $H$ in $(\Delta n / 2)(1+o(1))$ rounds. 

\medskip

In light of Observations~\ref{obs:cliques1} and~\ref{obs:cliques2}, it remains to investigate how large complete graphs one can build in $t$ rounds, provided that $t = t(n) = n^{1+o(1)}$ and $t = \bigo( n \log n )$. Recall that $f(n)=o(1)$ only means that $\lim_{n\to\infty} f(n)=0$ and so it might be negative. In particular, we allow $t \ll n$ as long as $t = n^{1+o(1)}$. If $t = o(n \log n)$, then one may create a complete graph of size which is asymptotic to the maximum number of squares on a single vertex---see Lemma~\ref{lem:squares}(b). More importantly, asymptotically, this is the best one can do. This is our first main result which we state below.

\begin{theorem}\label{thm:small_t}
Suppose that $t = t(n)$ is such that $t=n^{1+o(1)}$ and $t \ll n \log n$. Let $\beta = \beta(n) \ :=\ n \log n / t$. (In particular, $\beta \to \infty$, as $n \to \infty$.) 
Define 
$$
\ell \ =\  \ell(n) \ :=\  \frac {\log n}{\log \beta - 2 \log \log \beta}  \ \sim\  \frac {\log n}{\log \beta},
$$
and
$$
\epsilon \ =\  \epsilon(n) \ :=\  
\begin{cases}
e^2 (\log \beta) / \beta & \text{ if }\quad \beta \le \log n / \log \log n, \\
15 (\log \ell) / \ell & \text{ if }\quad \log n / \log \log n < \beta \le \log^2 n,\\
e / \ell & \text{ if }\quad \beta > \log^2 n.
\end{cases}
$$
(In particular, $\epsilon = o(1)$, regardless of $\beta$.) 
Finally, let
\begin{eqnarray*}
k &=& k(n) \ :=\ \frac { \log n - 2 \log \log n - t/n }{\log \beta} \ \sim\  \frac {\log n}{\log \beta} \\
k' &=& k'(n) \ :=\  \ell (1 + 4 \epsilon^{1/4})  \ \sim\  \frac {\log n}{\log \beta} \ \sim\  k.
\end{eqnarray*}
 
Then, the following hold.
\begin{itemize}
\item[(a)] There exists a strategy that a.a.s.\ creates $K_k$ at time $t$.
\item[(b)] There is no strategy that a.a.s.\ creates $K_{k'}$ at time $t$.
\end{itemize}
\end{theorem}

\medskip

Unfortunately, when $t = \Theta(n \log n)$, then our bounds do not asymptotically match but they are at most a multiplicative factor of $2+o(1)$ away from each other, as we will show later (see Figure~\ref{fig:cliques_ratio}). Suppose that $t=t(n) = \gamma n \log n$ for some $\gamma \in (0, \infty)$. We will derive an asymptotic lower bound of $k_1 = \ell = \xi \gamma \log n$, where constant $\xi = \xi(\gamma) \in (1,\infty)$ is defined to be the unique solution to the following equation 
\begin{equation}\label{eq:xi}
1 - \xi \gamma ( \log \xi - 1 ) - \gamma = 0,
\end{equation}
which is equivalent to 
\begin{equation}\label{eq:xi2}
\xi (\log \xi - 1) = \frac {1-\gamma}{\gamma} = \frac {1}{\gamma} - 1 \in (-1, \infty)
\end{equation}
or to
$$
\xi \gamma = \frac {1-\gamma}{\log \xi - 1}.
$$
(The left hand side of~(\ref{eq:xi2}) is a bijection from $(0,\infty)$ to $(-1, \infty)$ which proves the uniqueness of $\xi$.) As we will see in Lemma~\ref{lem:squares}(c) below, $\ell$ defined in Theorem~\ref{thm:small_t} is asymptotic to the maximum number of squares on a single vertex. 

It is easy to see that $\xi$ is a decreasing function of $\gamma$. If $\gamma \to 0$, then $\xi \sim (1/\gamma) / \log(1/\gamma)$, which is consistent with Theorem~\ref{thm:small_t} (applied with $\beta = 1/\gamma \to \infty$). If $\gamma = 1$, then $\xi = e$. More importantly, if $\gamma = \gamma_{\ell} = (2 \log 2-1)^{-1} \approx 2.59$, then $\xi = 2$. Finally, if $\gamma \to \infty$, then $\xi \to 1$. The constant $\gamma_{\ell}$ will play a special role in the lower bound in the statement of our result. We will show another asymptotic lower bound of $k_2 \sim 2 \gamma \log n$ which is stronger than the previous one, provided that $\gamma > \gamma_{\ell}$ (see Figure~\ref{fig:cliques}, right side). 

We will also show two upper bounds. The first one, $k'_2 \sim 2\ell \sim 2 \xi \gamma \log n$ is trivial but is best possible when $\gamma \to \infty$ (recall that $\xi \to 1$ as $\gamma \to \infty$). The second one, $k'_1 \sim (1+2\sqrt{2} (e/\xi)^{1/4}) \ell$, is stronger provided that $\gamma < \gamma_u = (64 e \log 64-1)^{-1} \approx 0.00139$ (see Figure~\ref{fig:cliques}, left side). Indeed, if $\gamma < \gamma_u$, then $\xi > 64 e$ and, as a consequence, $1+2\sqrt{2} (e/\xi)^{1/4} < 2$. This bound is best possible when $\gamma \to 0$.
 
\begin{theorem}\label{thm:large_t}
Suppose that $t = t(n) = \gamma n \log n$ for some $\gamma \in (0, \infty)$. Let $\xi = \xi(\gamma) \in (1,\infty)$ be defined as in~(\ref{eq:xi}). 
Define 
$$\ell \ =\  \ell (n)\ :=\ \frac {1-\gamma}{\log \xi - 1} \log n \ =\ \xi \gamma \log n .$$

Let
\begin{eqnarray*}
k_1 &=& k_1(n)  \ :=\ \frac { (1-\gamma) \log n - 2 \log \log n }{\log \xi - 1} \ \sim\ \ell \ = \  \xi \gamma \log n, \\
k_2 &=& k_2(n)  \ :=\  2 \gamma \log n - 4 \sqrt{ \gamma \log n \log \log n } \ \sim\  2 \gamma \log n, \\
k'_1 &=& k'_1(n) \ :=\  \left( 1+ \frac {3}{\log^{1/2} n} \right) \Big(1+2\sqrt{2} (e/\xi)^{1/4}\Big) \xi \gamma \log n \ \sim\  \Big( 1+2\sqrt{2} (e/\xi)^{1/4} \Big) \xi \gamma \log n, \\
k'_2 &=& k'_2(n)  \ :=\ 2 \xi \gamma \log n + 1 \ \sim\  2 \xi \gamma \log n. 
\end{eqnarray*}
Finally, let
\begin{eqnarray*}
k &=& k(n) \ :=\   \max \{k_1, k_2\}  \ \sim\  \max \{ \xi, 2 \} \gamma \log n, \\
k' &=& k'(n)  \ :=\   \min \{k'_1, k'_2\}  \ \sim\  \min \{ 1+2\sqrt{2} (e/\xi)^{1/4}, 2 \} \xi \gamma \log n.
\end{eqnarray*}

Then, the following hold.
\begin{itemize}
\item[(a)] There exists a strategy that a.a.s.\ creates $K_k$ at time $t$.
\item[(b)] There is no strategy that a.a.s.\ creates $K_{k'}$ at time $t$.
\end{itemize}
\end{theorem}

\begin{figure}[h]
     \centering
     \includegraphics[width=0.45\textwidth]{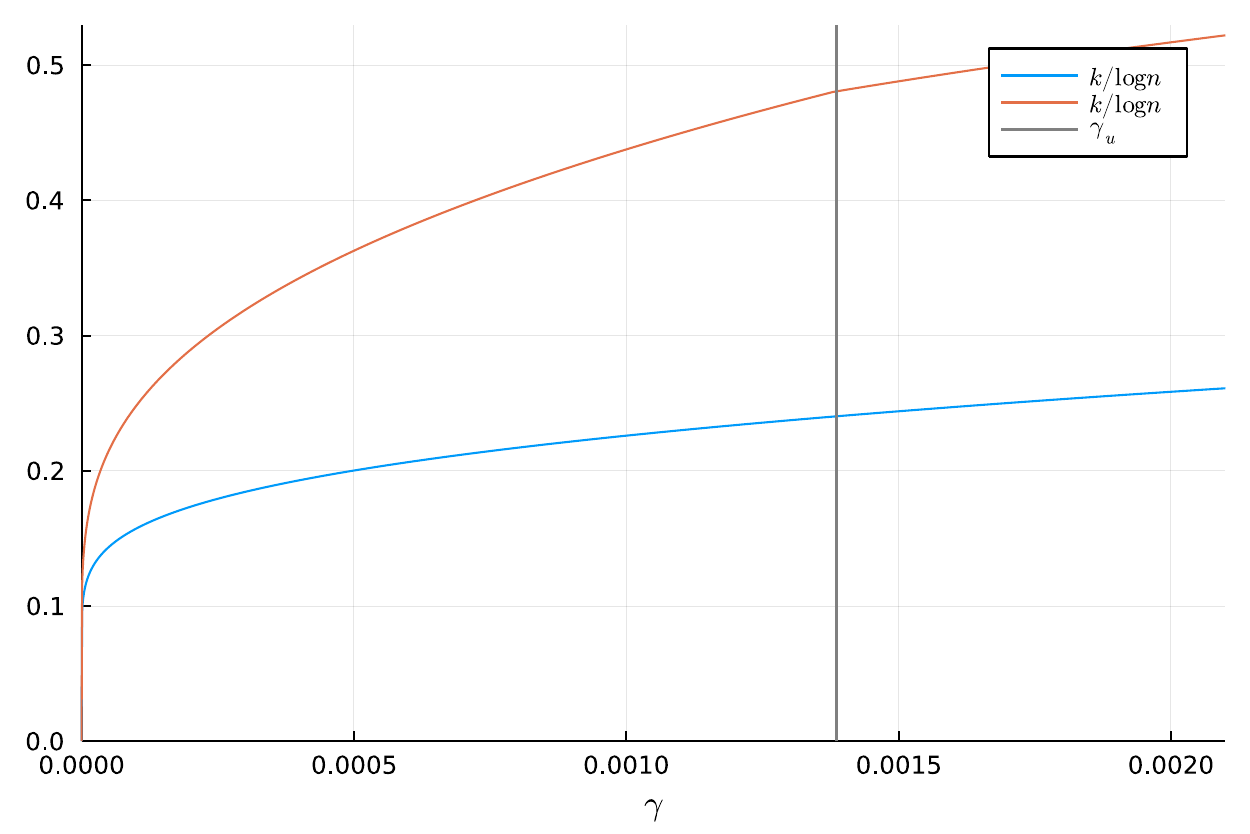}
     \includegraphics[width=0.45\textwidth]{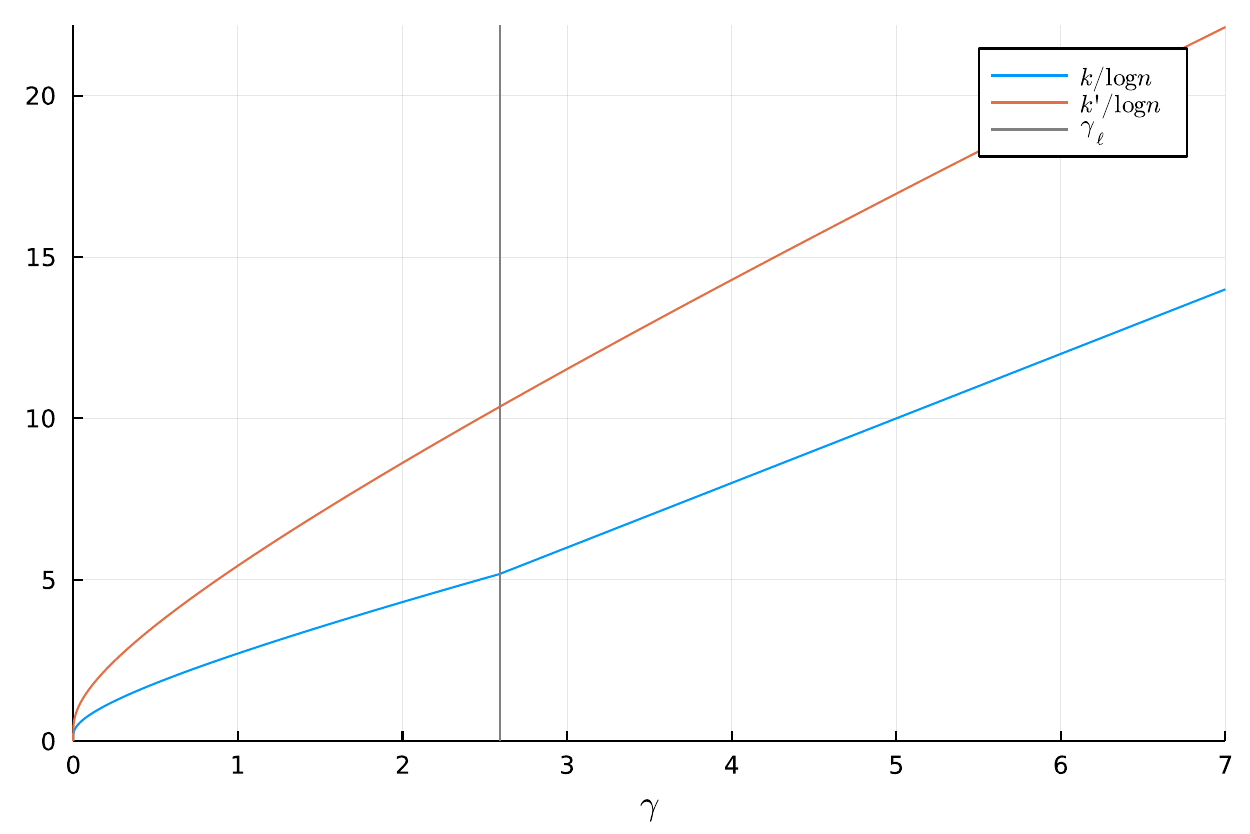}
     \caption{The upper ($k'$) and the lower ($k$) bound for the order of a largest complete graph: small (left figure) and large (right figure) values of $\gamma$.}
     \label{fig:cliques}
\end{figure}

\begin{figure}[h]
     \centering
     \includegraphics[width=0.45\textwidth]{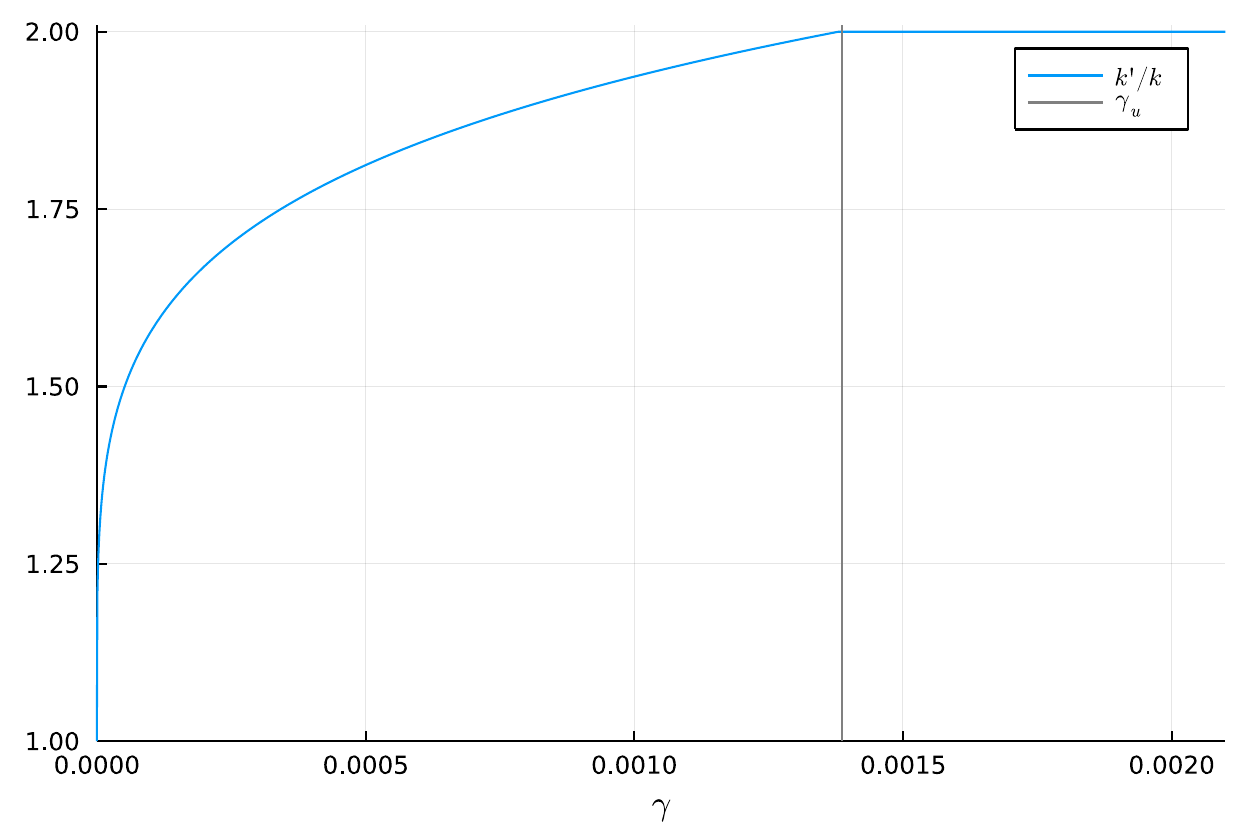}
     \includegraphics[width=0.45\textwidth]{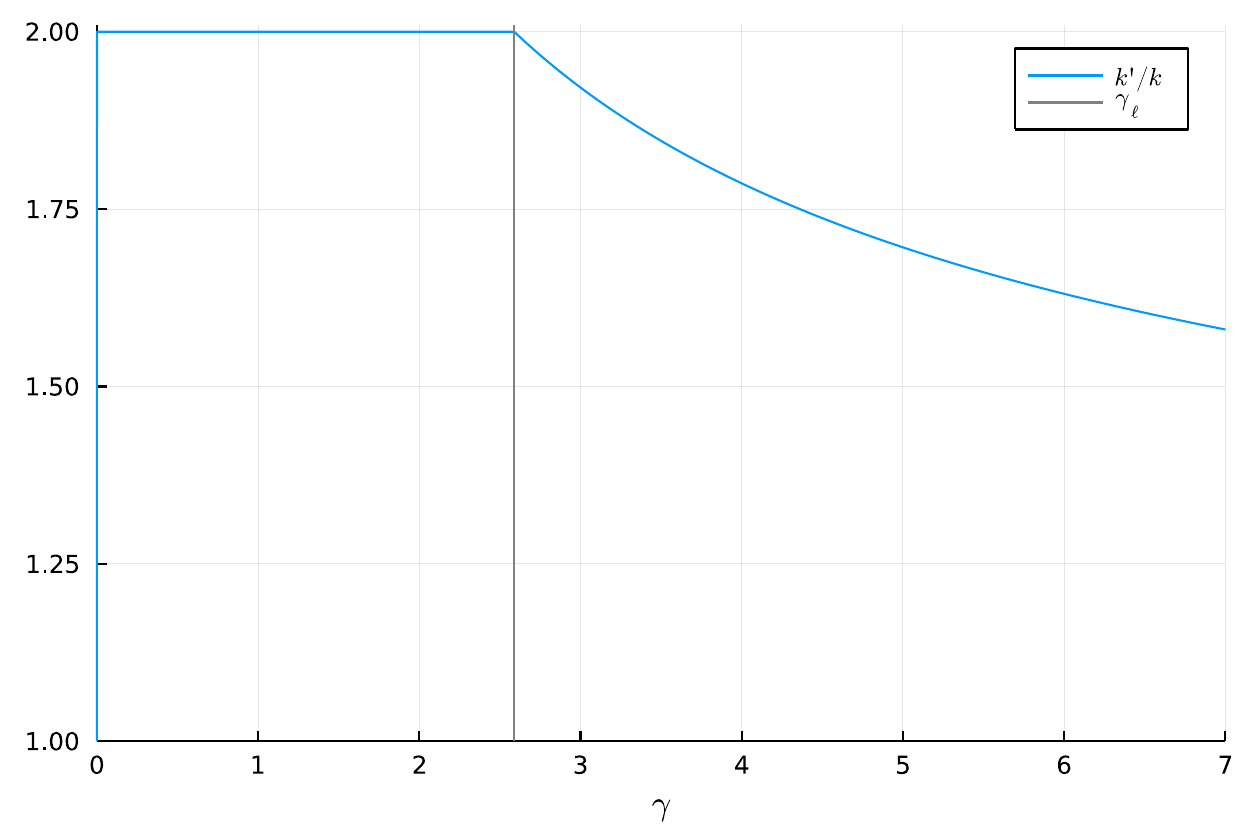}
     \caption{The ratio between the upper ($k'$) and the lower ($k$) bound for the order of a largest complete graph: small (left figure) and large (right figure) values of $\gamma$.}
     \label{fig:cliques_ratio}
\end{figure}

\subsection{Main Results---Chromatic Number} 

A \textbf{proper colouring} of a graph is a labeling of its vertices with colours such that no two vertices sharing the same edge have the same colour. The smallest number of colours in a proper colouring of a graph $G = (V, E)$ is called its \textbf{chromatic number}, and it is denoted by $\chi(G)$. Since this graph parameter is not well-defined for (multi)graphs with loops, we simply ignore them if they are present in $G_t$. Potential parallel edges do not cause any problems but, of course, can be ignored too.

The second monotone property we investigate in this paper is the property that $\chi(G_t) \ge k$ for some value of $k=k(n)$. Trivially, the player can achieve this property by constructing $K_k$ so earlier results immediately imply the corresponding lower bounds. We will prove matching upper bounds (up to a multiplicative factor of $2+o(1)$) yielding the following three results. Hence, in all regimes, the chromatic number is of order of the clique number.

In the first regime, the ratio between the upper and the lower bound is $2+o(1)$. 

\begin{theorem}\label{thm:chi_small_t}
Suppose that $t = t(n)$ is such that $t=n^{1+o(1)}$ and $t \ll n \log n$. Let $\beta = \beta(n) \ :=\ n \log n / t$. (In particular, $\beta \to \infty$, as $n \to \infty$.) 
Define $\ell = \ell(n)$ and $k = k(n)$ as in Theorem~\ref{thm:small_t}.
 
Then, the following hold.
\begin{itemize}
\item[(a)] There exists a strategy that a.a.s.\ creates $G_t$ such that $\chi(G_t) \ge k$.
\item[(b)] There is no strategy that a.a.s.\ creates $G_t$ such that $\chi(G_t) \ge 2 \ell + 2 \sim 2k$.
\end{itemize}
\end{theorem}

In the second regime, the ratio between the upper and the lower bound is at most $2+o(1)$ (see Figure~\ref{fig:chi}). 

\begin{theorem}\label{thm:chi_large_t}
Suppose that $t = t(n) = \gamma n \log n$ for some $\gamma \in (0, \infty)$. Let $\xi = \xi(\gamma) \in (1,\infty)$ be defined as in~(\ref{eq:xi}). 
Define $\ell = \ell(n)$ and $k = k(n)$ as in Theorem~\ref{thm:large_t}.

Then, the following hold.
\begin{itemize}
\item[(a)] There exists a strategy that a.a.s.\ creates $G_t$ such that $\chi(G_t) \ge k$.
\item[(b)] There is no strategy that a.a.s.\ creates $G_t$ such that $\chi(G_t) \ge 2 \ell + 2  = \Theta(k)$.
\end{itemize}
\end{theorem}

\begin{figure}[h]
     \centering
     \includegraphics[width=0.45\textwidth]{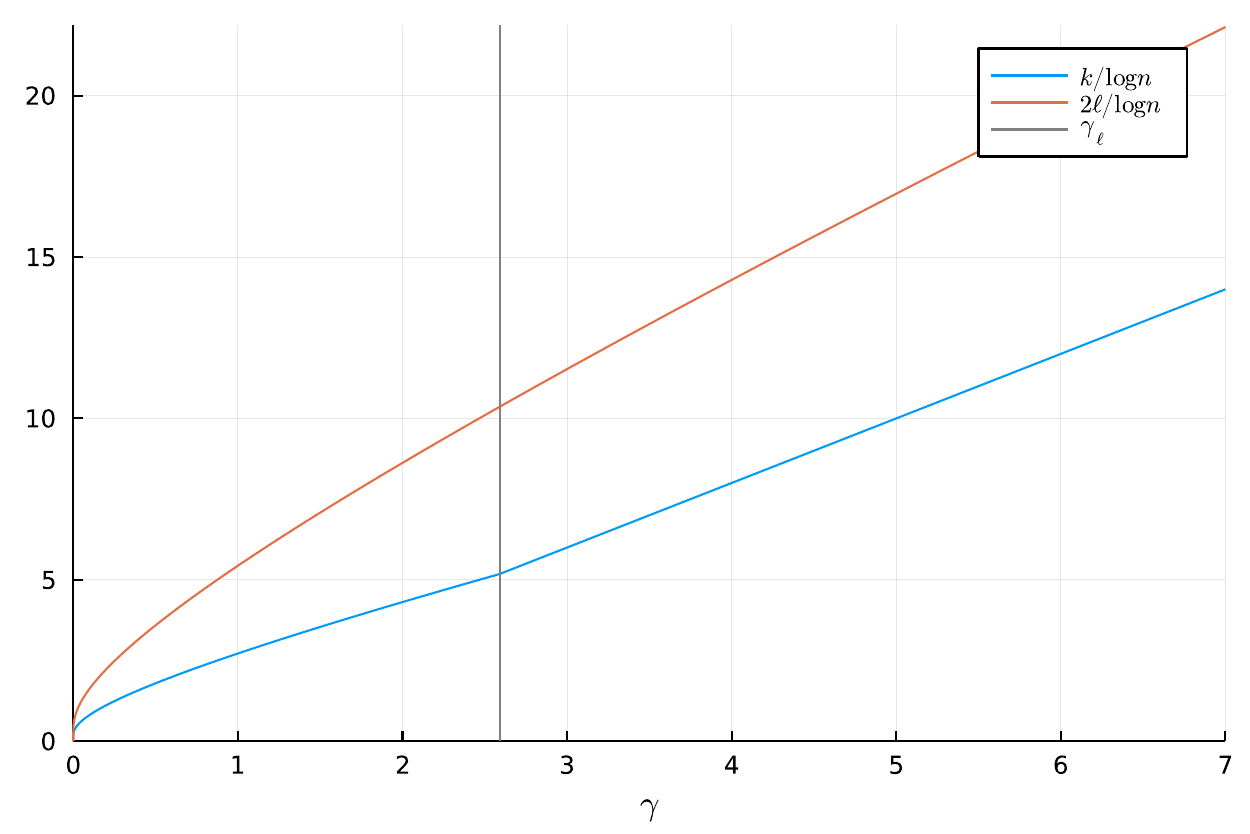}
     \includegraphics[width=0.45\textwidth]{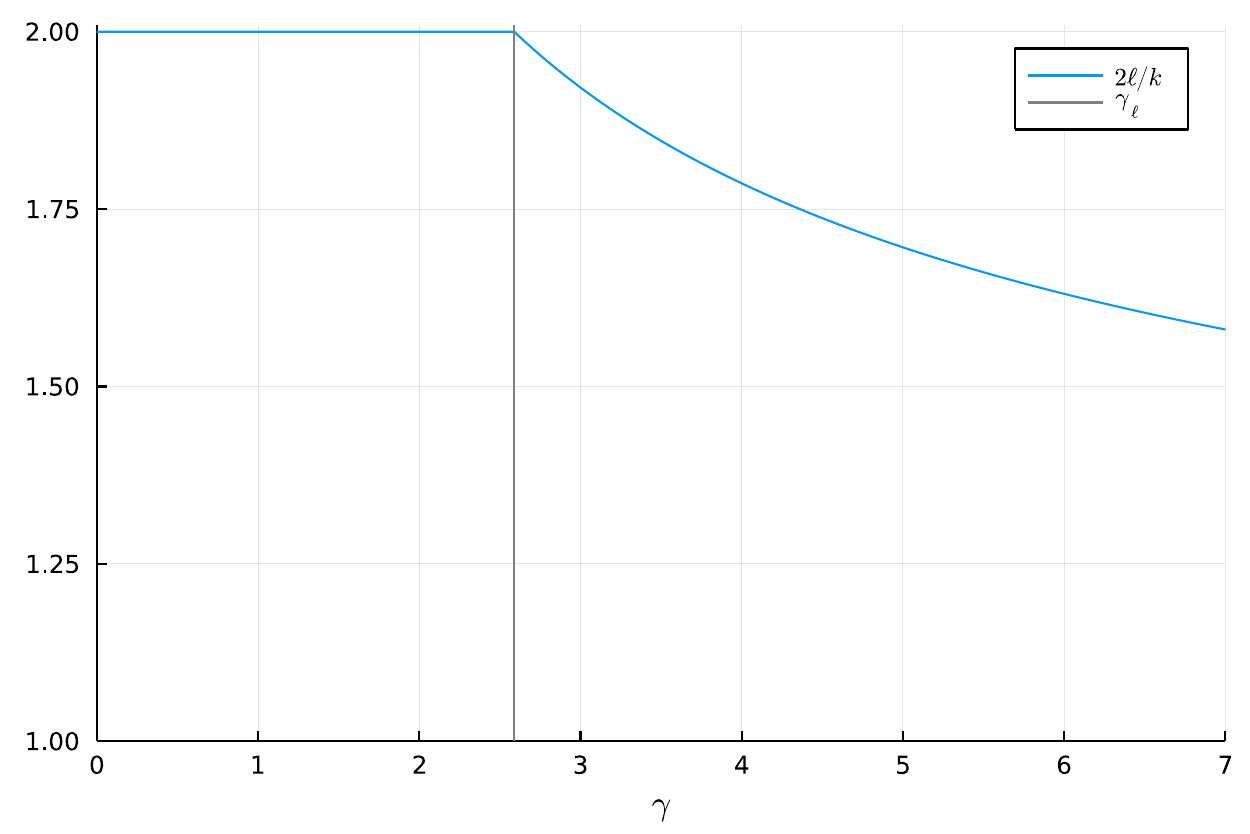}
     \caption{The upper ($2\ell+2$) and the lower ($k$) bound for $\chi(G_t)$ (left figure) as well as the ratio between the two (right figure).}
     \label{fig:chi}
\end{figure}

Note that if $\gamma \to \infty$ in the above result, then $\xi \to 1$ and so both bounds are asymptotically tight: $\chi(G_t) \sim 2 \gamma \log n$.

\begin{theorem}\label{thm:chi_very_large_t}
Suppose that $t = t(n) \ge \omega n \log n$, where $\omega=\omega(n)$ is any function that tends to infinity as $n \to \infty$. 
Define $k=k(n)$ and $k'=k'(n)$ as in Observation~\ref{obs:cliques2}.

Then, the following hold.
\begin{itemize}
\item[(a)] There exists a strategy that a.a.s.\ creates $G_t$ such that $\chi(G_t) \ge \min\{k,n\}$. 
\item[(b)] There is no strategy that a.a.s.\ creates $G_t$ such that $\chi(G_t) \ge k' \sim k$.
\end{itemize}
\end{theorem}

\subsection{Main Results---Independent Sets} 

An \textbf{independent set} is a set of vertices in a graph, no two of which are adjacent. The \textbf{independence number} $\alpha(G)$ of a graph $G=(V,E)$ is the cardinality of a maximum independent set of vertices. As for the chromatic number, we simply ignore loops if they are present in $G_t$.

The last monotone property we investigate in this paper is the property that $\alpha(G_t) \le k$ for a given value of $k=k(n)$. We have a good understanding of the independence number of $G_t$ when the average degree tends to infinity together with $n$.

\begin{theorem}\label{thm:alpha_large_t}
Suppose that $t = t(n)$ is such that $n \ll t \ll n^2$. Let $\lambda = \lambda(n) = t/n$. 

Then, the following hold.
\begin{itemize}
\item[(a)] There exists a strategy that a.a.s.\ creates $G_t$ such that 
$$\alpha(G_t) \ \le \ \frac {n}{2 \lambda} \left( 1 + \bigo(\sqrt{\log \lambda / \lambda}) + \bigo(\lambda/n) \right) \ \sim \ \frac {n}{2 \lambda}.$$
\item[(b)] There is no strategy that a.a.s.\ creates $G_t$ such that 
$$\alpha(G_t) \ < \ \frac {n}{2 \lambda + 1}.$$
\end{itemize}
\end{theorem}

Suppose now that the average degree is of the same order as the order of a graph, that is, $t=t(n) \sim c n^2$. In this case, we determine the independence number precisely unless $c = 1 / (2q)$ for some $q \in \N$. If $c = 1 / (2q)$ for some $q \in \N$, then the upper and the lower bounds may be off by one.

\begin{theorem}\label{thm:alpha_very_large_t}
Suppose that $t = t(n) \sim c n^2$ for some constant $c \in (0, 1]$. Let $\lambda = \lambda(n) = t/n \sim cn$. 

Then, the following hold.
\begin{itemize}
\item[(a)] There exists a strategy that a.a.s.\ creates $G_t$ such that 
$$\alpha(G_t) \ \le \ \left\lceil \frac {n}{2 \lambda} \left( 1 + \bigo(\sqrt{\log \lambda / \lambda}) \right) \right\rceil.$$
\item[(b)] There is no strategy that a.a.s.\ creates $G_t$ such that 
$$\alpha(G_t) \ < \ \frac {n}{2 \lambda + 1}.$$
\end{itemize}
\end{theorem}

On the other extreme case, if $t = t(n) \ll n$, then the number of vertices that are \emph{not} isolated is (deterministically) at most $2t = o(n)$ and so $\alpha(G_t) \sim n$. Understanding $\alpha(G_t)$ seems to be more challenging when $t \sim c n$ for some constant $c \in (0, \infty)$. It is easy to see that $\alpha(G_t) = \Theta(n)$ but determining the constants hidden in the $\Theta(\cdot)$ notation appears to be difficult. Indeed, in this regime we do not even know the behaviour of the independence number of the binomial random graph~\cite{bayati2010combinatorial}. This random graph, much easier model to analyze, is a special case of the semi-random process and can be easily mimicked by it. We provide a few natural upper and lower bounds later on but more work needs to be done to have a better understanding of this graph parameter. 

\subsection{Structure of the Paper}

We first introduce some basic concentration tools and present known results about the semi-random process (Section~\ref{sec:prelim}). Complete graphs are investigated in Section~\ref{sec:cliques}, chromatic number in Section~\ref{sec:chi}, and independent sets in Section~\ref{sec:alpha}. We conclude the paper with summarizing open problems that are left to be investigated (Section~\ref{sec:conclusions}). 

\section{Preliminaries}\label{sec:prelim}

\subsection{Concentration Tools}\label{sec:concentration}

Let us first state a few specific instances of Chernoff's bound that we will find useful. Let $X \in \textrm{Bin}(n,p)$ be a random variable distributed according to a Binomial distribution with parameters $n$ and $p$. Then, a consequence of \textbf{Chernoff's bound} (see e.g.,~\cite[Theorem~2.1]{JLR}) is that for any $t \ge 0$ we have
\begin{eqnarray}
\Prob( X \ge \E X + t ) &\le& \exp \left( - \frac {t^2}{2 (\E X + t/3)} \right)  \label{chern1} \\
\Prob( X \le \E X - t ) &\le& \exp \left( - \frac {t^2}{2 \E X} \right).\label{chern}
\end{eqnarray}

Moreover, let us mention that the bound holds in a more general setting as well, that is, for $X=\sum_{i=1}^n X_i$ where $(X_i)_{1\le i\le n}$ are independent variables and for every $i \in [n]$ we have $X_i \in \textrm{Bernoulli}(p_i)$ with (possibly) different $p_i$-s (again, see~e.g.,~\cite{JLR} for more details). Finally, it is well-known that the Chernoff bound also applies to negatively correlated Bernoulli random variables~\cite{dubhashi1998balls}. 

\subsection{The Differential Equation Method} \label{sec:DEs}

For one of our bounds, we will use the differential equation method (see~\cite{BD20} for a gentle introduction) to establish dynamic concentration of our random variables. The origin of the differential equation method stems from work done at least as early as 1970 (see Kurtz \cite{Kurtz1970}), and which was developed into a very general tool by Wormald \cite{W1995, W1999} in the 1990's. Indeed, Wormald proved a ``black box'' theorem, which gives dynamic concentration so long as some relatively simple conditions hold. Warnke \cite{Warnke2020} recently gave a short proof of a somewhat stronger black box theorem. 

\subsection{Literature Review} 

Since the semi-random process is still a relatively new model, let us highlight a few results on the model.

\subsubsection{Perfect Matchings} 

In the very first paper~\cite{beneliezer2019semirandom}, it was shown that the semi-random process is general enough to approximate (using suitable strategies) several well-studied random graph models, including an extensively studied $k$-out process (see, for example, Chapter~18 in~\cite{Karonski_Frieze}). In the $k$-out process, each vertex independently connects to $k$ randomly selected vertices which results in a random graph on $n$ vertices and $k n$ edges. 

Since the $2$-out process has a perfect matching \emph{a.a.s.}~\cite{frieze1986maximum}, we immediately get that one can create a perfect matching in $(2+o(1))n$ rounds. By coupling the semi-random process with another random graph that is known to have a perfect matching \emph{a.a.s.}~\cite{pittel}, the upper bound can be improved to $(1+2/e+o(1))n < 1.73576n$. This bound was consecutively improved by investigating a fully adaptive algorithm~\cite{gao2022perfect}. The currently best upper bound is $1.20524 n$. On the other hand, the lower bound observed in~\cite{beneliezer2019semirandom} ($(\ln(2)+o(1))n > 0.69314n$) was improved as well, and now we know that one needs at least $0.93261 n$ rounds to create a perfect matching~\cite{gao2022perfect}. 

\subsubsection{Hamilton Cycles} 

It is known that a.a.s.\ the famous $3$-out process is Hamiltonian~\cite{bohman2009hamilton}. Since the semi-random process can be coupled with the $k$-out process~\cite{beneliezer2019semirandom} (for any $k \in \N$), we get that a.a.s.\ one may create a Hamilton cycle in $(3+o(1)) n$ rounds. A new upper bound was obtained in~\cite{gao2020hamilton} in terms of an optimal solution to an optimization problem whose value is believed to be $2.61135 n $ by numerical support. 

The upper bound of $(3+o(1))n$ obtained by simulating the $3$-out process is \emph{non-adaptive}. That is, the strategy does \textit{not} depend on the history of the semi-random process. The above mentioned improvement proposed in~\cite{gao2020hamilton} uses an adaptive strategy but in a weak sense. The strategy consists of 4 phases, each lasting a linear number of rounds, and the strategy is adjusted \emph{only} at the end of each phase (for example, the player might identify vertices of low degree, and then focus on connecting circles to them during the next phase). 

In~\cite{gao2022fully}, a fully adaptive strategy was proposed that pays attention to the graph $G_t$ and the position of $u_t$ for every single step $t$. As expected, such a strategy creates a Hamilton cycle substantially faster than the weakly adaptive or non-adaptive strategies,
and it allows to improve the upper bound from $2.61135 n$ to $2.01678 n$. One more trick was observed recently which further improves an upper bound to $1.84887n$~\cite{frieze2022hamilton}. After combining all the ideas together, the currently best upper bound is equal to $1.81696 n$~\cite{frieze2023building}.

Let us now move to the lower bounds. As observed in the initial paper introducing the semi-random process~\cite{beneliezer2019semirandom}, if $G_t$ has a Hamiltonian cycle, then $G_t$ has minimum degree at least 2. Thus, a.a.s.\ it takes at least $( \ln 2+ \ln(1+\ln2) + o(1)) n \ge 1.21973 n$ rounds to achieve this property as this is exactly how many rounds are needed to get the property of having the minimum degree at least $2$~\cite{beneliezer2019semirandom}. In~\cite{gao2020hamilton}, the lower bound mentioned above was shown to not be tight. The lower bound was increased by $\eps n = 10^{-8} n$ and so numerically negligible. Better bound was obtained in~\cite{gao2022fully} and now we know that a.a.s.\ it takes at least $1.26575 n$ rounds to create a Hamilton cycle.

\subsubsection{Spanning Subgraphs} 

Let us now discuss what is known about the property of containing a given spanning graph $H$ as a subgraph. It was asked by Noga Alon whether for any bounded-degree $H$, one can construct a copy of $H$ \emph{a.a.s.}\ in $O(n)$ rounds.  This question was answered positively in a strong sense in~\cite{beneliezer2020fast}, in which it was shown that any graph with maximum degree $\Delta$ can be constructed \emph{a.a.s.}\ in $(3\Delta/2+o(\Delta))n$ rounds. They also proved that if $\Delta = \omega (\log(n))$, then this upper bound improves to $(\Delta/2 + o(\Delta))n$ rounds. Note that both of these upper bounds are asymptotic in $\Delta$. When $\Delta$ is constant in $n$, such as in both the perfect matching and Hamiltonian cycle setting, determining the optimal dependence on $\Delta$ for the number of rounds needed to construct $H$ remains open. Moving to this direction, $k$-factors and $k$-connectivity was studied recently in~\cite{koerts2022k}.

\subsubsection{A Few Other Directions} 

Finally, let us mention that sharp thresholds were studied recently in~\cite{macrury2022sharp}. In fact, the results from there apply to a more general class of processes including the semi-random process. Moreover, some interesting variants of the semi-random process are already considered. In~\cite{burova2022semi}, a random spanning tree of $K_n$ is presented to the player who needs to keep one of the edges. In~\cite{gilboa2021semi}, squares presented by the process follow a random permutation. In~\cite{Harjas}, the process in which $k$ random squares are presented and the player needs to select one of them before creating an edge is considered. Hypergraphs are investigated in~\cite{behague2022} and~\cite{molloy2023matchings}.

\section{Complete Graphs}\label{sec:cliques}

Let us start with a well-known observation, closely related to the coupon collector problem, that the number of squares on a given vertex is a random variable that is asymptotically distributed as the Poisson random variable. 

\begin{lemma}\label{lem:squares}
Suppose that $t=t(n)$ is such that $t = n^{1+o(1)}$ and $t = \bigo(n \log n)$. For any vertex $v \in [n]$, let $X_v=X_v(n)$ be the number of squares that land on $v$ until time $t$ and let 
$$\lambda \ :=\ \E[X_v] = t/n.$$
Then, the following properties hold.
\begin{itemize}
\item[(a)] For any vertex $v \in [n]$ and for any $k = o(\sqrt{n})$,
$$
\Prob( X_v = k ) \ \sim\ \frac {\lambda^k}{k!} \exp(-\lambda).
$$
\item [(b)] Suppose that $t = t(n) = n \log n / \beta$ for some $\beta = \beta(n) \to \infty$ as $n \to \infty$. As in Theorem~\ref{thm:small_t}, define 
$$\ell \ = \ \ell(n)\ :=\    \frac {\log n}{\log \beta - 2 \log \log \beta} \ \sim\ \frac {\log n}{\log \beta} \ =\ \frac {\beta}{\log \beta} \cdot \lambda.$$
Then, a.a.s.
$$
\max \Big( X_v : v \in [n] \Big) \ \le\ \ell \  \sim \ \frac {\beta}{\log \beta} \cdot \lambda.
$$
\item [(c)] Suppose that $t = t(n) = \gamma n \log n$ for some $\gamma \in (0,\infty)$. Let $\xi = \xi( \gamma ) \in (1,\infty)$ be as in~(\ref{eq:xi}). As in Theorem~\ref{thm:large_t}, define 
$$\ell \ = \ \ell(n)\ :=\   \frac { 1-\gamma }{ \log \xi - 1 } \log n \ =\ \xi \gamma \log n \ =\ \xi \lambda.$$
 Then, a.a.s.
$$
\max \Big( X_v : v \in [n] \Big) \ \le\ \ell \  = \ \xi \lambda.
$$
\end{itemize}
\end{lemma}

\begin{proof}
Note that for any $k = o(\sqrt{n})$,
\begin{eqnarray*}
\Prob( X_v = k ) &=& \binom{t}{k} \left( \frac {1}{n} \right)^k \left( 1 - \frac {1}{n} \right)^{t-k} \\
&=& \frac {t^k}{k!} \left( 1 + \bigo(k/t) \right)^k \left( \frac {1}{n} \right)^k \exp \left( - \frac {1}{n} + \bigo( 1/n^2 ) \right)^{t-k} \\
&=& \frac {(t/n)^k}{k!} \left( 1 + \bigo(k^2/t) \right) \exp \left( - t/n + \bigo(k/n) + \bigo( t/n^2 ) \right) \\
&\sim& \frac {\lambda^k}{k!} \exp(-\lambda).
\end{eqnarray*}
The property~(a) holds. 

\medskip

To show property~(b), first note that, since $t = n^{1+o(1)}$, we get that $\beta = n^{o(1)}$ (but also, by definition, $\beta \to \infty$ as $n \to \infty$) and so $\ell \to \infty$ as $n \to \infty$. It follows from part~(a) and the Stirling's formula ($\ell! \sim \sqrt{2 \pi \ell} (\ell/e)^{\ell}$) that for any vertex $v \in [n]$,
\begin{eqnarray*}
\Prob \Big( X_v = \ell \Big) &\sim& \frac {\lambda^\ell}{\ell!} e^{-\lambda} \le \frac {\lambda^\ell}{\ell!}  = o \left( \frac { (t/n)^\ell }{ (\ell/e)^\ell } \right) = o \left( \left( \frac {e \log n}{\ell \beta} \right)^{\ell} \right) = o \left( \left( (1+o(1)) \frac {e \log \beta}{\beta} \right)^{\ell} \right) \\
&=& o \left( \exp \left( - \ell \big( \log \beta - \log \log \beta - \bigo(1) \big) \right) \right) \\
&=& o \left( \exp \left( - \frac {\log n}{\log \beta - 2 \log \log \beta} \big( \log \beta - \log \log \beta - \bigo(1) \big) \right) \right) \\
&=& o \left( \exp \left( - \log n \right) \right) = o(1/n).
\end{eqnarray*}
Now, note that for any $\ell \le k < n^{1/3}$ we have
$$
\frac {\Prob(X_v=k+1)}{\Prob(X_v=k)} \sim \frac{\lambda}{k+1} \le \frac {\lambda}{\ell} \sim \frac {\log n / \beta}{\log n / \log \beta} = \frac {\log \beta}{\beta} = o(1),
$$
and so $\Prob(\ell \le X_v \le n^{1/3}) \sim \Prob(X_v=\ell)$. Finally, note that 
$$
\Prob(X_v \ge n^{1/3}) \le \binom{t}{n^{1/3}} \left( \frac {1}{n} \right)^{n^{1/3}} \le \left( \frac {et}{n^{4/3}} \right)^{n^{1/3}} = \exp \Big( - \Theta( n^{1/3} \log n ) \Big) = o(1/n).
$$
Combining all of these properties together we get that 
\begin{equation}\label{eq:prob_Xv}
\Prob(X_v \ge \ell) \le \Prob(\ell \le X_v \le n^{1/3}) + \Prob(X_v \ge n^{1/3}) = o(1/n),
\end{equation} 
and the property~(b) holds by the union bound over all vertices $v$.

\medskip

The same argument can be used to show property~(c). This time, for any vertex $v \in [n]$,
\begin{eqnarray*}
\Prob \Big( X_v = \ell \Big) &\sim& \frac {\lambda^\ell}{\ell!} e^{-\lambda} = o \left( \left( \frac {\gamma \log n}{\ell/e} \right)^{\ell} \exp \Big( - \gamma \log n \Big) \right) \\
&=& o \left( \exp \Big( \ell \log (e / \xi) - \gamma \log n \Big) \right) \\
&=& o \left( \exp \Big( - \ell ( \log \xi - 1) - \gamma \log n \Big) \right) \\
&=& o \left( \exp \Big( - (1-\gamma) \log n - \gamma \log n \Big) \right) \\
&=& o ( 1/n ).
\end{eqnarray*}
Since $\xi > 1$, for any $\ell \le k < n^{1/3}$ we have
$$
\frac {\Prob(X_v=k+1)}{\Prob(X_v=k)} \sim \frac{\lambda}{k+1} \le \frac {\lambda}{\ell} = \frac {1}{\xi} < 1,
$$
and so $\Prob(\ell \le X_v \le n^{1/3}) = \bigo( \Prob(X_v=\ell) )$. The conclusion is the same, $\Prob(X_v \ge \ell) = o(1/n)$ (see~(\ref{eq:prob_Xv})), and the property~(c) holds by the union bound over all vertices $v$. 
\end{proof}

We independently consider lower and upper bounds for the order of complete graphs one may build during the semi-random process. 

\subsection{Lower Bounds}

We present two algorithms that can be used to build complete graphs. Both can be used for any value of $t$ but the first algorithm will turn out to be better than the second one, provided that $t \le \gamma_{\ell} n \log n$, where $\gamma_{\ell} = (2 \log 2-1)^{-1} \approx 2.59$.

\begin{algorithms}\label{alg:1}
The algorithm consists of $k-1$ phases. 
The first phase has only one round in which the player creates $K_2$, a single edge. 

At the beginning of phase $i$, $i \ge 2$, a complete graph $K_i$ on the vertex set $v_1, v_2, \ldots, v_i$ is already constructed. Any other vertex $v$ that is covered by $s=s(v) \ge 1$ squares has the following property: for any $j \in [s]$, the $j$-th square is connected to a circle on vertex $v_j$. The player maintains this property by applying the following strategy. If a square lands on vertex $v \notin \{v_1, v_2, \ldots, v_i\}$ that is already covered by $s$ squares, then she connects $v$ to vertex $v_{s+1}$. (If $v = v_j$ for some $j \in [i]$, then she plays arbitrarily---that edge is ignored anyway.) The $i$-th phase ends when a square lands on a vertex with $i-1$ squares and a copy of a complete graph $K_{i+1}$ is created.

The algorithm ends at the end of phase~$k-1$ when a copy of $K_k$ is created.
\end{algorithms}

The analysis of Algorithm~\ref{alg:1} proves Theorem~\ref{thm:small_t}(a) and the first part of Theorem~\ref{thm:large_t}(a), namely, the lower bound of $k_1$.

\begin{proof}[Proof of Theorem~\ref{thm:small_t}(a) and the first part of Theorem~\ref{thm:large_t}(a)]
Let us first prove Theorem~\ref{thm:small_t}(a). Recall that
$$
k = k(n) = \frac { \log n - 2 \log \log n - \lambda }{\log \beta} \sim \frac {\log n}{\log \beta} = \frac {\beta}{\log \beta} \cdot \lambda,
$$
where $\lambda = \lambda (n) = t / n \ll \log n$ is the expected number of squares on a given vertex and $\beta = \beta(n) = n \log n / t \to \infty$, as $n \to \infty$. Note that $\beta = n^{o(1)}$ (but also, by definition, $\beta \to \infty$ as $n \to \infty$) and so $k \to \infty$ as $n \to \infty$. On the other hand, since $\beta \to \infty$ as $n \to \infty$, we get that $k = o(\log n)$.

Suppose that the player uses Algorithm~\ref{alg:1} to play the game against the semi-random process. We will show that at time $t$ a.a.s.\ there are at least $k$ vertices that are covered by at least $k$ squares. It is easy to see that the algorithm has to end before that round. This will imply the lower bound of $k$. 

Let $Y = Y(t)$ be the number of vertices that are covered by $k$ squares at time $t$. Using Lemma~\ref{lem:squares} and the Stirling's formula ($k! \sim \sqrt{2\pi k} (k/e)^k$), we get that
\begin{eqnarray}
\E[Y] &\sim& n \cdot \frac{\lambda^k}{k!} \exp(-\lambda) \nonumber \\
&=& \Theta(k^{-1/2}) \cdot n \cdot \left( \frac {e\lambda}{k} \right)^k \exp(-\lambda) \nonumber \\
&=& \Theta(k^{-1/2}) \cdot \exp \left( \log n - k \log \left( \frac {k}{e\lambda} \right) -\lambda \right). \label{eq:exp_y}
\end{eqnarray}
Since $k \le \frac {\beta}{\log \beta} \cdot \lambda$, we get that $\frac {k}{e\lambda} \le \frac {\beta}{e \log \beta} \le \beta$ and so 
\begin{eqnarray*}
\E[Y] &=& \Omega(k^{-1/2}) \cdot \exp \left( \log n - k \log \beta -\lambda \right) \\
&=& \Omega(k^{-1/2}) \cdot \exp \left( \log n - (\log n - 2 \log \log n - \lambda) -\lambda \right) \\
&=& \Omega \left( \frac {\log^2 n}{k^{1/2}} \right) \gg k^{3/2} \gg k,
\end{eqnarray*}
as $k = o(\log n)$. Finally, note that events ``vertex $v$ is covered by $k$ squares'' associated with different vertices are negatively correlated. As a result, it follows immediately from Chernoff's bound~(\ref{chern}) and the comment right after it that a.a.s.\ $Y \ge k$. (Alternatively, one could us the second moment method to get the same conclusion.) As argued above, this finishes the proof of Theorem~\ref{thm:small_t}(a).

\medskip

The same argument implies the lower bound of $k_1$ in Theorem~\ref{thm:large_t}(a). Recall that this time $\lambda = \lambda (n) = t / n = \gamma \log n$ and
$$
k_1 = k_1(n) = \frac { (1-\gamma) \log n - 2 \log \log n }{\log \xi - 1} \sim \frac {1-\gamma}{\log \xi - 1} \log n = \xi \gamma \log n.
$$
Computations performed in~(\ref{eq:exp_y}) still apply. Since $k_1 \le \xi \gamma \log n$, we get that 
\begin{eqnarray*}
\E[Y] &=& \Omega(k_1^{-1/2}) \cdot \exp \left( \log n - k_1 \log (\xi/e) -\gamma \log n \right) \\
&=& \Omega(k^{-1/2}) \cdot \exp \left( \log n - ((1-\gamma) \log n - 2 \log \log n) -\gamma \log n \right) \\
&=& \Omega \left( \frac {\log^2 n}{k^{1/2}} \right) = \Omega \left( k^{3/2} \right) \gg k,
\end{eqnarray*}
as $k = \Theta(\log n)$. As before, we conclude that a.a.s.\ $Y \ge k$ and we are done.
\end{proof}

To prove the second part of Theorem~\ref{thm:large_t}(a) (that is, the lower bound of $k_2$), we need to analyze the second algorithm which performs better when $t \ge \gamma_{\ell} n \log n$.

\begin{algorithms}\label{alg:2}
Suppose that $k=2\ell+1$ for some $\ell \in \N$. Before the game starts, select arbitrarily $k$ vertices $v_0, v_1, \ldots, v_{2\ell}$ from the vertex set $[n]$. For each $i \in [2\ell] \cup \{0\}$ and for each $j \in [\ell]$, the player connects the $j$th square landing on vertex $v_i$ with vertex $v_{(i+j) \pmod{2\ell}}$. The algorithm ends when each vertex $v_i$  is covered by at least $\ell$ squares, and a copy of $K_k$ is created.
\end{algorithms}

Clearly, the player could partition the vertex set into $n/k$ sets and then use a slightly more sophisticated algorithm in which she tries to simultaneously create $n/k$ complete graphs, by applying Algorithm~\ref{alg:2} independently to each part. Such algorithm would finish once at least one set induces a complete graph. We do not use and analyze it as it would give an asymptotically negligible improvement on the lower bound. However, it will be useful later on once we analyze the chromatic number.

\begin{proof}[Proof of the second part of Theorem~\ref{thm:large_t}(a)]
Recall that $t = t(n) = \gamma n \log n$ for some $\gamma \in (0, \infty)$ and
$$
k_2 = k_2(n) = 2 \gamma \log n - 4 \sqrt{ \gamma \log n \log \log n } \sim 2 \gamma \log n.
$$
Let $k$ be the smallest odd integer that is at least $k_2$ and assume $k=2\ell+1$ for some $\ell \in \N$. Clearly, $k = k_2 + \bigo(1)$.

Suppose that the player uses Algorithm~\ref{alg:2} to play the game against the semi-random process. For any $i \in [2\ell] \cup \{0\}$, let $X_i$ be the number of squares on $v_i$ at time $t$. Note that $X_i \in \textrm{Bin}(t,1/n)$ with $\E[X_i] = \gamma \log n$. It follows from Chernoff's bound~(\ref{chern}) that 
\begin{eqnarray*}
\Prob( X_i < \ell ) &=& \Prob \Big( X_i \le \E [X_i] - 2 \sqrt{ \gamma \log n \log \log n } + \bigo(1) \Big) \\
&\le& \exp \left( - \frac { (2 \sqrt{ \gamma \log n \log \log n } + \bigo(1))^2 }{ 2 \gamma \log n } \right) \\
&=& \exp \Big( - 2 \log \log n + o(1) \Big) \\
&\sim& (\log n)^{-2}.
\end{eqnarray*}
Hence, by the union bound over all vertices $v_i$, the algorithm does not finish before time $t$ with probability at most $(2\ell+1) \cdot \bigo( (\log n)^{-2} ) = \bigo( (\log n)^{-1} ) = o(1)$. Hence the desired bound holds a.a.s., and the proof is finished. 
\end{proof}

\subsection{Upper Bounds}

Suppose first that $t = t(n) = n \log n / \beta$ for some $\beta = \beta(n) \to \infty$ as $n \to \infty$. Lemma~\ref{lem:squares}(b) shows that a.a.s.\ the number of squares on any vertex is at most $\ell \sim \log n / \log \beta$. It implies immediately that one cannot construct a complete graph of size larger than $2\ell+1$. But the truth is that a.a.s.\ it is only possible to build a complete graph of size asymptotic to $\ell$. The main difficulty in showing this lies in the fact that the player can easily create vertices of large degree by placing a large number of circles on some vertices. So this is certainly not the bottleneck for this problem. The key observation used in the proof is that in order to create a large complete graph, many squares need to land on vertices that are already covered by circles, but this happens quite rarely. 

Suppose that vertex $u$ is covered by $k$ squares, $u_{t_1}, u_{t_2}, \ldots, u_{t_k}$, for some $k \le \ell$. We will first estimate the number of squares that land ``on top'' of the associated circles $v_{t_1}, v_{t_2}, \ldots, v_{t_k}$. Formally, we will say that $(v_{t_i}, u_{s})$ is a \textbf{rare pair} if $v_{t_i}$ and $u_{s}$ cover the same vertex and $s > t_i$. Note, in particular, that when distinct squares fall on the same circle, those are still counted as distinct rare pairs. The name is justified by the next lemma which shows that on average only $o(\ell)$ squares land on a given circle.

\begin{lemma}\label{lem:rare_pairs}
Suppose that $t = t(n)$ is such that $t=n^{1+o(1)}$ and $t \ll n \log n$. Let $\beta = \beta(n) = n \log n / t \to \infty$, as $n \to \infty$. Let $\ell = \ell(n) \sim \log n / \log \beta$ and $\epsilon = \epsilon (n) = o(1)$ be defined as in Theorem~\ref{thm:small_t}. Then, the following property holds a.a.s.: for any vertex $u$,
\begin{itemize}
\item[(a)] $u$ is covered by $k \le \ell$ squares, $u_{t_1}, u_{t_2}, \ldots, u_{t_k}$,
\item[(b)] the associated circles, $v_{t_1}, v_{t_2}, \ldots, v_{t_k}$, belong to less than $\ell^2 \epsilon = o(\ell^2)$ rare pairs.
\end{itemize}
\end{lemma}

\begin{proof}
Property~(a) follows immediately from Lemma~\ref{lem:squares}(b). Since we aim for a statement that holds a.a.s.\ we may assume that property~(a) holds when proving property~(b). 

Note that if $(v_{t_i}, u_{s})$ forms a rare pair, then the square $u_{s}$ needs to arrive after the circle $v_{t_i}$ is placed (that is, $s>t_i$). Hence, the probability that a given vertex $u$ fails to satisfy property~(b) does not depend on the strategy of the player and can be upper bounded by 
$$
p \ :=\ \binom{t}{\ell^2 \epsilon} \left( \frac {\ell}{n} \right)^{\ell^2 \epsilon} \le \left( \frac {et}{\ell^2 \epsilon} \cdot \frac {\ell}{n} \right)^{\ell^2 \epsilon} = \left( \frac {e \log n}{\ell \beta \epsilon} \right)^{\ell^2 \epsilon} = \exp \left( - \ell^2 \epsilon \log \left( \frac {\ell \beta \epsilon}{e \log n} \right) \right).
$$

\medskip

Suppose first that $\beta \le \log n / \log \log n$. Then, since $\ell \ge \log n / \log \beta$, we get that $\epsilon = e^2 \log \beta / \beta \ge e^2 \log n / (\beta \ell)$ and so
$$
p \le \exp \left( - \ell^2 \epsilon \right) \le \exp \left( - e^2 \ell \log n / \beta \right).
$$
Since $\beta \le \log n / \log \log n$ (so, in particular, $\log \beta \le \log \log n$),
$$
p \le \exp \left( - e^2 \ell \log \log n \right) \le \exp \left( - e^2 \ell \log \beta \right) \le \exp \left( - e^2 \log n \right) = o(1/n).
$$

\medskip

Suppose now that $\log n / \log \log n < \beta \le \log^2 n$. In particular, $(1+o(1)) \log \log n \le \log \beta \le 2 \log \log n$. Then, since 
$$
\log \ell \ge \log \log n - \log \log \beta = (1+o(1)) \log \log n,
$$
we get that 
$$
\epsilon = \frac {15 \log \ell} {\ell} \ge \frac {2 e^2 \log \log n}{\ell} \ge \frac {e^2 \log \beta}{\ell} = \frac {e^2 \ell \log \beta}{\ell^2} \ge \frac {e^2 \log n}{\ell^2}.
$$
It follows that 
$$
p \le \exp \left( - e^2 \log n \cdot \log \left( \frac {e \beta}{\ell} \right) \right) = \exp \left( - e^2 \log n \cdot \log \left( (e+o(1)) \frac {\beta \log \beta}{\log n} \right) \right).
$$
Since $\beta \ge \log n / \log \log n$ and $\log \beta \ge (1+o(1)) \log \log n$, we conclude that
$$
p \le \exp \left( - e^2 \log n \cdot \log \Big( (e+o(1)) \Big) \right) = \exp \left( - (e^2 + o(1)) \log n \right) = o(1/n).
$$

\medskip

Finally, suppose that $\beta > \log^2 n$. This time $\epsilon = e / \ell$ and, since $\sqrt{\beta} > \log n$,
\begin{eqnarray*}
p &\le& \exp \left( - e \ell \log \left( \frac {\beta}{\log n} \right) \right) \le \exp \left( - e \ell \log \left( \sqrt{\beta} \right) \right) = \exp \Big( - (e/2) \ell \log \beta \Big) \\
&=& \exp \Big( - (e/2) \log n \Big) = o(1/n).
\end{eqnarray*}
In all scenarios, the conclusion follows by the union bound over all vertices. 
\end{proof}

For a given graph $G=(V,E)$ and any set of vertices $S \subseteq V$, we will use $G[S]$ to denote the graph \textbf{induced} by set $S$, that is, $G[S] = (S, E')$ and edge $uv \in E$ is in $E'$ if and only if $u \in S$ and $v \in S$. The above lemma immediately implies the following useful corollary. 

\begin{corollary}\label{cor:rare_pairs}
Suppose that $t = t(n)$ is such that $t=n^{1+o(1)}$ and $t \ll n \log n$.  Let $\ell = \ell(n)$ and $\epsilon = \epsilon(n)$ be defined as in Lemma~\ref{lem:rare_pairs}. Then, a.a.s.\ for any set $S \subseteq [n]$, $G_t[S]$ has at most $|S| \ell^2 \epsilon$ rare pairs.
\end{corollary}

Our next observation shows that one can remove only a few edges from $G_t[S]$ in order to destroy all rare pairs.

\begin{lemma}\label{lem:destroying_rare_pairs}
Suppose that $t = t(n)$ is such that $t=n^{1+o(1)}$ and $t \ll n \log n$. Let $\ell = \ell(n)$ and $\epsilon = \epsilon(n)$ be defined as in Theorem~\ref{thm:small_t}. Then, a.a.s.\  for any set $S \subseteq [n]$, one can remove at most $2 |S| \ell \sqrt{\epsilon}$ edges from $G_t[S]$ to remove all rare pairs.
\end{lemma}

\begin{proof}
Suppose that some vertex $w$ yields $r$ rare pairs. Let 
$$
R_w = \{ (v_{t_i}, u_{s_i}) : v_{t_i} = u_{s_i} = w \text{ and } s_i > t_i \}
$$
be the set of  rare pairs associated with vertex $w$, let $V_w = \{ v_{t_i} : (v_{t_i}, u_{s_i}) \in R_w \text{ for some } u_{s_i} \}$ be the set of the associated circles, and let $U_w = \{ u_{s_i} : (v_{t_i}, u_{s_i}) \in R_w \text{ for some } v_{t_i} \}$ be the set of the associated squares. We will first show that one can remove at most $2 \sqrt{r}$ edges to destroy all rare pairs from $R_w$. 

If $|V_w| \le \sqrt{r}$, then one can remove all edges associated with circles from $V_w$ which clearly destroys all rare pairs from $R_w$. Similarly, if  $|U_w| \le \sqrt{r}$, then one can achieve the same by removing all edges associated with squares from $U_w$. We may then assume that $|V_w| > \sqrt{r}$ and that $|U_w| > \sqrt{r}$.

Let $\hat{s}$ be the largest integer from $[t]$ with the property that $\hat{U}_w = \{ u_{s_i} \in U_w : s_i \ge \hat{s} \}$ has cardinality $\sqrt{r}$. In other words, $\hat{U}_w \subseteq U_w$ consists of the $\sqrt{r}$ ``youngest'' squares from $U_w$. Similarly, let $\hat{t}$ be the smallest integer from $[t]$ with the property that $\hat{V}_w = \{ v_{t_i} \in V_w : t_i \le \hat{t} \}$ has cardinality $\sqrt{r}$. In other words, $\hat{V}_w \subseteq V_w$ consists of the $\sqrt{r}$ ``oldest'' circles from $V_w$. Let us remove all edges associated with squares from $\hat{U}_w$, and let us remove all edges associated with circles from $\hat{V}_w$, (or both), for a total of at most $2 \sqrt{r}$ edges. We claim that this procedure removes all rare pairs from $R_w$. 

For a contradiction, suppose that a rare pair $(v_{t_i}, u_{s_i})$ is not removed. In particular, $t_i > \hat{t}$ and $s_i < \hat{s}$. On the other hand, by the definition of being a rare pair, $t_i < s_i$. We conclude that $\hat{t} < \hat{s}$. But it means that each circle from $\hat{V}_w$ forms a rare pair with any square from $\hat{U}_w$ for a total of $\sqrt{r} \cdot \sqrt{r} = r$ rare pairs. With the additional rare pair $(v_{t_i}, u_{s_i})$, there are at least $r+1$ rare pairs which gives us the desired contradiction, and the claim is proved: one can remove at most $2 \sqrt{r}$ edges to destroy all rare pairs from $R_w$. 

\medskip

Consider any set $S \subseteq [n]$. By Corollary~\ref{cor:rare_pairs}, since we aim for a statement that holds a.a.s., we may assume that $G_t[S]$ yields at most $|S| \ell^2 \epsilon$ rare pairs, that is, $\sum_{w \in S} r_w \le |S| \ell^2 \epsilon$, where $r_w$ is the number of rare pairs associated with vertex $w$. By the above observation, one may remove at most $\sum_{w \in S} 2 \sqrt{r_w}$ edges from $G_t[S]$ to destroy all of them. Clearly, the optimization problem
$$
\max \sum_{w \in S} 2 \sqrt{r_w} \qquad \text{ subject to } \qquad \sum_{w \in S} r_w \le |S| \ell^2 \epsilon \text{ and } r_w \ge 0 \text{ for all } w \in S
$$
attains its maximum when all $r_w$'s are equal. We conclude that it is possible to remove at most 
$$
\sum_{w \in S} 2 \sqrt{r_w} \le \sum_{w \in S} 2 \sqrt{ \ell^2 \epsilon } = 2 |S| \ell \sqrt{\epsilon}
$$
edges from $G_t[S]$ to destroy all rare pairs, and the proof of the lemma is finished.
\end{proof}

Now, we are ready to finish the proof of Theorem~\ref{thm:small_t}.

\begin{proof}[Proof of Theorem~\ref{thm:small_t}(b)] 
Let us apply \emph{any} strategy to play the game. For a contradiction, suppose that there exists a $K_{k'}$ at time $t$, where $k' = \ell (1+4\epsilon^{1/4})$ as in the statement of the theorem. Recall that $\ell \sim \log n / \log \beta$ and $\epsilon = \epsilon(n) = o(1)$. It is also straightforward to check that $\ell \epsilon^{1/4}\to\infty$ as $n \to \infty$. Let $S \subseteq [n]$ be any set of cardinality $k'$ that induces a complete graph. Since we aim for a statement that holds a.a.s., we may apply Lemma~\ref{lem:destroying_rare_pairs}. It follows that one can remove at most $2k'\ell\epsilon^{1/2} < 4 \ell^2 \epsilon^{1/2}$ edges from $G_t[S]$ in order to destroy all rare pairs. Note that after that operation, set $S$ satisfies the following properties:
\begin{itemize}
\item [(a)] $S$ has cardinality $k' = \ell (1+4\epsilon^{1/4})$,
\item [(b)] the number of edges induced by $S$ (and so also the number of squares) is more than
$$
\binom{k'}{2} - 4 \ell^2 \epsilon^{1/2} = \binom{\ell}{2} + 4 \ell^2 \epsilon^{1/4} + \binom{4 \ell \epsilon^{1/4}}{2} - 4 \ell^2 \eps^{1/2} > \binom {\ell}{2} + 4 \ell^2 \epsilon^{1/4},
$$
(note that the first equality is due to a simple fact that $\binom{a+b}{2}= \binom{a}{2} + ab + \binom{b}{2}$ for $a,b\in \mathbb N$)
\item [(c)] $S$ induces no rare pair,
\item [(d)] there are at most $\ell$ squares on any vertex.
\end{itemize}

Let us now remove all edges induced by $S$ and put them back, one by one, following the order they appeared during the semi-random process. We will distinguish $k'$ phases. The first phase starts when the circle associated with the first edge lands on vertex $v_1 \in S$. Since there are no rare pairs (property~(c)), no square will land on $v_1$ but other circles might end up there. The first phase continues as long as circles continue landing on $v_1$. The second phase starts when some circle lands on vertex $v_2 \neq v_1$. Note that, since all edges introduced during the first phase are edges of the graph and all the circles cover vertex $v_1$, all squares cover unique vertices at the beginning of the second phase. In particular, there is at most one square on vertex $v_2$ when the first circle is placed on $v_2$. 

In general, phase~$i$ starts when some circle is placed on vertex $v_i \not\in \{v_1, v_2, \ldots, v_{i-1}\}$. Arguing as above, we conclude that at that point there are at most $i-1$ squares on $v_i$, and no more squares can land on it in the future. This is a useful upper bound for the number of squares, provided that $i \le \ell$. For larger values of $i$ (that is, $\ell < i \le k' = \ell (1+4\epsilon^{1/4})$), we may apply property~(d). We conclude that the number of squares is at most $\binom{\ell}{2} + 4\epsilon^{1/4} \ell \cdot \ell$, which contradicts property~(b). This finishes the proof of the theorem.
\end{proof}

\medskip

Suppose now that $t=t(n) = \gamma n \log n$ for some $\gamma \in (0,\infty)$. The upper bound of $k'_2$ in Theorem~\ref{thm:large_t}(b) is trivial and the argument above can be easily adjusted to show the upper bound of $k'_1$. We carefully explain the adjustment needed below.

\begin{proof}[Proof of Theorem~\ref{thm:large_t}(b)] 
First, let us note that the upper bound of $k'_2 = 2\ell+1$ is indeed trivial. For a contradiction, suppose that some set $S$ of cardinality at least $2\ell+2$ induces a complete graph on $(2\ell+2)(2\ell+1)/2$ edges (and so it induces that many squares). Hence, by averaging argument, there is a vertex with at least $(2\ell+1)/2 > \ell$ squares which contradicts Lemma~\ref{lem:squares}(c).

\medskip

It remains to prove the upper bound of $k'_1$. As in Lemma~\ref{lem:rare_pairs}, we first need to upper bound the number of rare pairs generated by one vertex. The probability that a given vertex $u$ generates at least $c \ell^2$ rare pairs is at most
\begin{eqnarray*}
p &:=& \binom {t}{c\ell^2} \left( \frac {\ell}{n} \right)^{c\ell^2} \le \left( \frac {e \gamma n \log n}{c\ell^2} \cdot \frac {\ell}{n} \right)^{c \ell^2} = \left( \frac {e \gamma \log n}{c\ell} \right)^{c \ell^2} = \left( \frac {e}{c\xi} \right)^{c \ell^2} \\
&=& \exp \left( - c (\xi \gamma \log n)^2 \log \left( \frac {c\xi}{e} \right) \right) \\
&=& \exp \left( - \Theta( \log^2 n) \log \left( 1 + \frac {1}{\sqrt{\log n}} \right) \right) \\
&=& \exp \left( - \Theta( \log^{3/2} n) \right) = o (1/n),
\end{eqnarray*} 
when $c = (e/\xi) (1+\log^{-1/2} n)$. (Note that $\log(1+x) = x+\bigo(x^2)$.) By the union bound over all vertices $u$ we get that a.a.s.\ no vertex generates at least $c \ell^2$ rare pairs. We conclude that a.a.s.\ for any set $S \subseteq [n]$, $G_t[S]$ induces at most $c \ell^2 |S|$ rare pairs (the counterpart of Corollary~\ref{cor:rare_pairs}). Lemma~\ref{lem:destroying_rare_pairs} still applies and we get that a.a.s.\ for any set $S$, one can remove at most $2 \sqrt{c} \ell |S|$ edges from $G_t[S]$ to remove all rare pairs. 

We finish the proof as before. For a contradiction, suppose that there exists a complete graph on $k' = \ell (1 + b c^{1/4})$ vertices, where $b$ is a constant that will be properly tuned soon. We may assume that $1 + b c^{1/4} \le 2$; otherwise, the trivial upper bound of $k'_2 = 2\ell+1$ applies. A.a.s.\ for any set $S$ of cardinality $k'$, after destroying all rare pairs from $G_t[S]$, the number of edges left is at least
\begin{eqnarray*}
\binom{k'}{2} - 2 \sqrt{c} \ell k'
&=& 
\binom{\ell}{2} + b c^{1/4} \ell^2 + \binom{b c^{1/4} \ell}{2} - 2 \sqrt{c} \ell k' \\
&>& \binom {\ell}{2} + b c^{1/4} \ell^2 + \frac {b^2 c^{1/2} \ell^2}{2} \left( 1 - \bigo\left( \frac {1}{\ell} \right) \right) - 4 \sqrt{c} \ell^2 \\
&=& \binom {\ell}{2} + b c^{1/4} \ell^2 + \frac {c^{1/2} \ell^2}{2} \left( b^2 - 8 - \bigo\left( \frac {1}{\log n} \right) \right) \\
&>& \binom {\ell}{2} + b c^{1/4} \ell^2,
\end{eqnarray*} 
when, for example, $b = \sqrt{8} (1 + \log^{-1/2} n)$. As before, we get the desired contradiction which implies the upper bound of
\begin{eqnarray*}
k' = \Big(1 + b c^{1/4}\Big) \ell &=& \Big(1 + 2 \sqrt{2} (e/\xi)^{1/4} (1 + \log^{-1/2} n)^2 \Big) \ell \\
&\le& \Big(1 + 2 \sqrt{2} (e/\xi)^{1/4} \Big) \ell \Big(1 + 3 \log^{-1/2} n \Big) = k'_1.
\end{eqnarray*} 
This finishes the proof of the theorem.
\end{proof}

\section{Chromatic Number}\label{sec:chi}

Parts~(a) in Theorems~\ref{thm:chi_small_t}, \ref{thm:chi_large_t} and~\ref{thm:chi_very_large_t} follow immediately from our results for complete graphs, namely, parts (a) in Theorems~\ref{thm:small_t} and~\ref{thm:large_t}, and Observation~\ref{obs:cliques2}. Parts~(b) will follow from upper bounds for the number of squares that land on vertices, Lemma~\ref{lem:squares}.

\medskip

Let us start with some useful basic facts about degeneracy of graphs. Recall that for a given $d \in \N$, a graph $H$ is \textbf{$d$-degenerate} if every sub-graph $H'\subseteq H$ has minimum degree $\delta(H') \le d$ (where the minimum degree of a graph is the minimum degree over all vertices). The \textbf{degeneracy} of $H$ is the smallest value of $d$ for which $H$ is $d$-degenerate.

The \textbf{$d$-core} of a graph $H$ is the maximal induced subgraph $H'\subseteq H$ with minimum degree $\delta(H')\ge d$. (Note that the $d$-core is well defined, though it may be empty. Indeed, if $S \subseteq V(H)$ and $T \subseteq V(H)$ induce sub-graphs with minimum degree at least $d$, then the same is true for $S \cup T$.) If $H$ has degeneracy $d$, then it has a non-empty $d$-core. Indeed, by definition, $H$ is \emph{not} $(d-1)$-degenerate and so it has a sub-graph $H'$ with $\delta(H')\ge d$. Moreover, it follows immediately from the definition that if $H$ has degeneracy $d$, then there exists a permutation of the vertices of $H$, $(v_1, v_2, \ldots , v_k)$, such that for each $\ell \in [k]$ vertex $v_\ell$ has degree at most $d$ in the sub-graph induced by the set $\{v_1, v_2, \ldots, v_{\ell}\}$. Indeed, one can define such permutation recursively. Let $v_k$ be any vertex in $H$ that is of degree at most $d$. Then, let $v_{k-1}$ be any vertex of degree at most $d$ in the graph $H'$ induced by the set $\{v_1, v_2, \ldots, v_{k-1}\}$, etc.

The above properties imply a useful reformulation of degeneracy: a graph $H$ is $d$-degenerate if and only if the edges of $H$ can be oriented to form a directed acyclic graph $D$ with maximum out-degree at most $d$. In other words, there exists a permutation of the vertices of $H$, $(v_1, v_2, \ldots, v_k)$, such that for every directed edge $(v_i,v_j)\in D$ we have $i > j$ and the out-degrees are at most $d$. As a consequence, we get another well-known but useful property: for any $d$-degenerate graph $H$ we have $\chi(H) \le d+1$. Indeed, one may colour vertices of $H$ greedily using the permutation $(v_{k}, v_{k-1}, \ldots, v_1)$.

\medskip

With these properties, we maye easily prove Theorems~\ref{thm:chi_small_t} and~\ref{thm:chi_large_t}.

\begin{proof}[Proof of Theorems~\ref{thm:chi_small_t}(b) and~\ref{thm:chi_large_t}(b)]
Fix an arbitrary strategy for the player, and consider graph $G_t$ generated at time $t$. Let $X_v=X_v(n)$ be the number of squares that land on $v$ until time $t$.
By Lemma~\ref{lem:squares}, we know that a.a.s.\ 
$$
\max \Big( X_v : v \in [n] \Big) \ \le\ \ell.
$$
Since we aim for a statement that holds a.a.s., we may assume that this property is satisfied. Let $S \subseteq [n]$ be any subset of vertices. Since each edge connects a square with a circle, $G_t[S]$ induces at most $|S|\ell$ edges and so the average degree in $G_t[S]$ is at most $2\ell$. It follows that $\delta(G_t[S]) \le 2\ell$ for any $S\subseteq [n]$ and so $G_t$ is $2\ell$-degenerate. The above observation implies that $\chi(G_t) \le 2\ell +1$ which finishes the proof of the theorem.
\end{proof}

\section{Independent Sets}\label{sec:alpha}

\subsection{Upper Bound}

We will first prove an upper bound that not only implies Theorem~\ref{thm:alpha_large_t}(a) and Theorem~\ref{thm:alpha_very_large_t}(a) but also provides a good upper bound when the average degree of $G_t$ is a constant, especially when that constant is large. Having said that, we do not tune our argument to get the best possible bound but rather aim for an easy argument that provides the upper bound that matches the lower bound when the average degree tends to infinity as $n \to \infty$. 

\begin{lemma}
Suppose that $t = t(n) = \Omega(n)$. Let $\lambda = \lambda(n) = t/n$, let $\ell = \ell(n) = \lambda - \sqrt{5 \lambda \log \lambda}$, and let $k = k(n) = 2 \lceil \ell \rceil + 1$. Finally, let 
\begin{itemize}
\item $u = u(n) = \lceil n/k \rceil = \left\lceil \frac {n}{2\lambda} (1 + \bigo( \sqrt{\log \lambda / \lambda} ) ) \right\rceil$, if $\lambda \gg n^{2/5}$,
\item $u = u(n) = \lceil n/k \rceil ( 1 + k^2 \sqrt{\log \lambda}/\lambda^{5/2} ) = \frac {n}{2\lambda} (1 + \bigo( \sqrt{\log \lambda / \lambda} ) )$, if $1 \ll \lambda = \bigo( n^{2/5} )$,
\item $u = u(n) = n \left( \frac {1}{2 \lceil \ell \rceil + 1}  + \frac {2 \lceil \ell \rceil}{\lambda^{5/2}} \right) + n^{3/4}$, if $\lambda = \Theta(1)$.
\end{itemize}
Then, there exists a strategy that a.a.s.\ creates $G_t$ such that $\alpha(G_t) \le u$. 
\end{lemma}

\begin{proof}
Let us arbitrarily partition the set of vertices $[n]$ into $\lceil n/k \rceil$ parts, each of size at most $k = 2 \lceil \ell \rceil + 1$. We will independently apply Algorithm~\ref{alg:2} to each part. The algorithm succeeds on a given part and produces a complete graph of order $k$ if all vertices in that part receive at least $\ell$ squares at time $t$. For a given $i \in \{1,2, \ldots, \lceil n/k \rceil \}$, let $X_i$ be the indicator random variable for the event that Algorithm~\ref{alg:2} fails on part $i$, and let $X = \sum_{i=1}^{\lceil n/k \rceil} X_i$ be the number of parts that failed. 

For a given vertex $v \in [n]$, let $Y_v$ be the random variable counting the number of squares on $v$ at time $t$. Clearly, $Y_v \in  \textrm{Bin}(t,1/n)$ with $\E [Y_v] = t/n = \lambda$. It follows immediately from Chernoff's bound~(\ref{chern}) that 
$$
\Prob ( Y_v < \ell ) = \Prob \Big( Y_v <  \lambda - \sqrt{5 \lambda \log \lambda} \Big) \le \exp \left( - \frac {5 \lambda \log \lambda}{2 \lambda} \right) = 1 / \lambda^{5/2}.
$$
Since the algorithm fails on a given part if at least one vertex in that part (out of at most $k$ vertices) receives less than $\ell$ squares at time $t$, 
$$
\Prob ( X_i = 1 ) \le k / \lambda^{5/2} = \bigo ( 1 / \lambda^{3/2} ).
$$
Hence, the expected number of parts that fail can be estimated as follows:
$$
\E [ X ] \le \lceil n/k \rceil \cdot \frac {k}{\lambda^{5/2}} = \bigo ( n / \lambda^{5/2} ).
$$

If $\lambda \gg n^{2/5}$, then $\E[X] \to 0$ as $n \to \infty$ and so, by Markov's inequality, we get that a.a.s.\ $X = 0$. Since each independent set can have at most one vertex from each part (as all of them are successful a.a.s.), we get that a.a.s.\ $\alpha(G_t) \le \lceil n/k \rceil$ and the desired bound holds. 

Suppose then that $1 \ll \lambda = \bigo (n^{2/5})$. This time, by Markov's inequality we get that a.a.s.\ $X \le \E[X] \sqrt{\log \lambda} = \bigo ( n \sqrt{\log \lambda} / \lambda^{5/2} )$. As before, each independent set can have at most one vertex from each successful part and, trivially, at most $k X$ vertices from parts that failed. We get that a.a.s.
\begin{eqnarray*}
\alpha(G_t) &\le& \left\lceil \frac {n}{k} \right\rceil + k X = \frac {n}{2\lambda} (1 + \bigo( \sqrt{\log \lambda / \lambda} ) ) + \bigo ( n \sqrt{\log \lambda} / \lambda^{3/2} )  \\
&=& \frac {n}{2\lambda} (1 + \bigo( \sqrt{\log \lambda / \lambda} ) ).
\end{eqnarray*}

Finally, suppose that $\lambda = \Theta(1)$. Note first that $(X_i)$ is a sequence of negatively correlated random variables. Combining this with earlier observations, we conclude that $X$ can be stochastically upper bounded by $Y = \sum_{i=1}^{\lceil n/k \rceil} Y_i$, where $(Y_i)$ are negatively correlated Bernoulii random variables with parameter $k/\lambda^{5/2}$. It follows from Chernoff's bound~(\ref{chern1}) and the comment right after it that a.a.s.\ $X \le Y \le \frac {n}{\lambda^{5/2}} + n^{2/3}$. We use the same observation as before, namely, that each independent set can have at most one vertex from each of the $\lceil n/k \rceil - X$ successful parts and at most $k X$ vertices from the $X$ parts that failed. We get that a.a.s.
\begin{eqnarray*}
\alpha(G_t) &\le& \left( \left\lceil \frac {n}{k} \right\rceil - X \right) + k X = \left\lceil \frac {n}{k} \right\rceil + (k-1) X \\
&\le& n \left( \frac {1}{2 \lceil \ell \rceil + 1}  + \frac {2 \lceil \ell \rceil}{\lambda^{5/2}} \right) + O(n^{2/3}).
\end{eqnarray*}
This finishes the proof of the lemma.
\end{proof}

Let us note that the upper bounds we just proved are asymptotically tight when $\lambda = \lambda(n) \gg 1$. On the other hand, the above bound is not sharp when $\lambda = \Theta(1)$. There are many ways one may improve it. For example, a more careful argument could estimate the size of a largest independent set of a given part that fails (right now, we simply use a trivial upper bound of $k$). Moreover, some parts that succeed receive more squares than needed for the argument (for example, perhaps each vertex in that part receives more than $\ell$ squares). Such additional squares could be used to create slightly larger cliques. Finally, when $\lambda$ is a small constant, a better strategy could be to create a large perfect matching using the adaptive algorithm analyzed in~\cite{gao2022perfect}.

\subsection{Lower Bounds}

It is easy to prove by induction that for any graph $G=(V,E)$, $\alpha(G) \ge n / (\Delta + 1)$, where $\Delta$ is the maximum degree. Caro~\cite{caro1979new} and Wei~\cite{wei1981lower} independently proved the following, more refined, version of this observation. (See also~\cite{alon2016probabilistic}.)

\begin{observation}[\cite{caro1979new,wei1981lower}]\label{obs:ind_set}
For any graph $G=(V,E)$,
$$
\alpha(G) \ge \sum_{v \in V} \frac {1}{\deg(v)+1} \ge \frac {n}{d+1} \ge \frac {n}{\Delta+1},
$$
where $d = \frac {1}{n} \sum_{v \in V} \deg(v)$ is the average degree and $\Delta = \max_{v \in V} \deg(v)$ is the maximum degree. 
\end{observation}

This observation immediately proves parts (b) of Theorems~\ref{thm:alpha_large_t} and~\ref{thm:alpha_very_large_t} since the average degree of $G_t$ is $d = 2 |E(G_t)| / n = 2 t / m = 2 \lambda$. As mentioned above, this simple observation is asymptotically tight when $\lambda = \lambda(n) \gg 1$. 

\medskip

We propose two improvements when $\lambda = \Theta(1)$. Observation~\ref{obs:ind_set} still applies but the lower bound of $n/(d+1)=n/(2\lambda+1)$ that holds deterministically may be improved with slightly more effort. Indeed, the existence of an independent set of size 
$$
L(G_t) := \sum_{v \in V} \frac {1}{\deg(v)+1}
$$ 
is still deterministically guaranteed. Understanding the graph parameter $\alpha(G_t)$ is challenging but $L(G_t)$ is relatively easy to deal with. In order to minimize $L(G_t)$, the player should keep the degree distribution as ``flat'' as possible; in particular, note that $L(G_t)$ is minimized when all vertices have degrees $\lfloor d \rfloor$ or $\lceil d \rceil$. However, she cannot achieve such distribution since squares arrive uniformly at random and so a.a.s.\ some vertices will receive more than $\lceil d \rceil$ squares. 

\medskip

To get a weaker lower bound we may consider an ``off-line'' version of the semi-random process, that is, let the player wait till time $t$ before placing all of her circles at once. Clearly, the original process (the ``on-line'' version) is at least as challenging to the player as its off-line counterpart, so the obtained lower bound also applies there and the lower bound for $\alpha(G_t)$ holds. 

Let $Y_k$ be the number of vertices that received $k$ squares at time $t$. It is easy to see (see Lemma~\ref{lem:squares}) that a.a.s.\ for any $k \in \N \cup \{0\}$,
$$
Y_k \ \sim\ n \frac {\lambda^k}{k!} \exp(-\lambda).
$$
The player will place her $t = \lambda n$ circles greedily on vertices with minimum degree. Let $M \in \N \cup \{0\}$ be the largest integer $m$ such that
$$
f(m) := \sum_{k=0}^{m-1} (m-k) \frac {\lambda^k}{k!} \exp(-\lambda) \le \lambda.
$$
For each $k \in \{0, 1, \ldots, M-1\}$, the player may put $M-k$ circles on each vertex with $k$ squares to make them of degree $M$. The total number of circles used so far is $(1+o(1)) n f(M)$ and the fraction of vertices of degree $M$ at this point is asymptotic to
$$
g(M) := \sum_{k=0}^M \frac {\lambda^k}{k!} \exp(-\lambda). 
$$
The remaining $(1+o(1)) n (\lambda - f(M))$ circles are places on such vertices. Once this is done there are $(1+o(1)) n h(k)$ vertices of degree $k$, where 
$$
h(k) = 
\begin{cases}
g(M) - (\lambda - f(M)) = g(M) - \lambda + f(M) & \text{ if } \qquad k=M \\
\frac {\lambda^{M+1}}{(M+1)!} \exp(-\lambda) + \lambda - f(M) & \text{ if } \qquad k=M+1 \\
\frac {\lambda^k}{k!} \exp(-\lambda) & \text{ if } \qquad k \ge M+2.
\end{cases}
$$
These observations imply the following lower bound.

\begin{lemma}
Suppose that $t = t(n) \sim \lambda n$ for some $\lambda \in (0,\infty)$. Let $\epsilon > 0$ be any (arbitrarily small) constant. Then, there is no strategy that a.a.s.\ creates $G_t$ such that 
$$
\alpha(G_t) < (1-\epsilon) \, n \sum_{k \ge M} \frac {h(k)}{k+1}.
$$ 
\end{lemma}

\medskip

Finally, we analyze how small $L(G_t)$ can get for the original (``on-line'') semi-random process. It is easy to see that in order to minimize $L(G_t)$ the player needs to apply a greedy strategy. In this strategy, in each round $s \le t$ of the process, the player puts a circle on a vertex with minimum degree; if there is more than one such vertex to choose from, the decision which one to select is made arbitrarily.

Note that in each round $s \le t$, the minimum degree in $G_s$ is at most the average degree, that is, at most $2s/n \le  2t/n = 2 \lambda$ so the player will never put a circle on a vertex of degree more than $2 \lambda$. In the greedy strategy, we distinguish phases by labelling them with integers $q \in \{0, 1, \ldots, r\}$, where $r = \lfloor 2 \lambda \rfloor$. During the $q$th phase, the minimum degree in $G_s$ is equal to $q$. In order to analyze the evolution of the semi-random process, we will track the following sequence of $r+1$ variables: for $0 \le i \le r$, let $Y_i = Y_i(s)$ denote the number of vertices in $G_s$ of degree $i$. 

Phase~$0$ starts at the beginning of the process. Since $G_0$ is empty, $Y_0(0) = n$ and $Y_i(0) = 0$ for $1 \le i \le r$. There are initially many isolated vertices but they quickly disappear. Phase~$0$ ends at time $s$ which is the smallest value of $s$ for which $Y_0(s) = 0$. The DEs method (see Subsection~\ref{sec:DEs}) will be used to show that a.a.s.\ Phase~$0$ ends at time $s_0 \sim x_0 n$, where $x_0$ is an explicit constant which will be obtained by investigating the associated system of DEs. Moreover, the number of vertices of degree $i$ ($1 \le i \le r$) at the end of this phase is well concentrated around some values that are also determined based on the solution to the same system of DEs: a.a.s.\ $Y_i(s_0) \sim y_i(x_0) n$. With that knowledge, we move on to Phase~$1$ in which we prioritize vertices of degree 1. 

Consider any Phase~$q$, where $q \in \{0, 1, \ldots, r\}$. This phase starts at time $s_{q-1}$, exactly when the previous phase ends (or at time $s_{-1} := 0$ if $q=0$). At that point, the minimum degree of $G_{s_{q-1}}$ is $q$, so $Y_i(s) = 0$ for any $s \ge s_{q-1}$ and $i < q$. Hence, we only need to track the behaviour of the remaining $r+1-q$ variables. Let us denote $H(s) = (Y_q(s), Y_{q+1}(s), \ldots, Y_{r}(s))$. Let $\delta_A$ be the Kronecker delta for the event $A$, that is, $\delta_A = 1$ if $A$ holds and $\delta_A=0$ otherwise. Then, for any $i$ such that $q \le i \le r$, 
\begin{align}\label{eq:min_degree_trend} 
\E \Big( Y_i(s+1) - Y_i(s) ~|~ H(s) \Big) &= -\delta_{i=q} + \delta_{i=q+1} - \frac {Y_i(s)}{n} + \delta_{i \ge q+1} \cdot \frac {Y_{i-1}(s)}{n}.
\end{align}
Indeed, since the circle is put on a vertex of degree $q$, we always lose one vertex of degree $q$ (term $-\delta_{i=q}$) that becomes of degree $q+1$ (term $\delta_{i=q+1}$). We might lose a vertex of degree $i$ when the square lands on a vertex of degree $i$ (term $Y_i(s)/n$). We might also gain one of them when the square lands on a vertex of degree $i-1$ (term, $Y_{i-1}(s)/n$); note that this is impossible if $i = q$ (term $\delta_{i \ge q+1}$). This suggests the following system of DEs: for any $i$ such that $q \le i \le r$, 
\begin{align}
y'_i(x) &= -\delta_{i=q} + \delta_{i=q+1} - y_i(x) + \delta_{i \ge q+1} \cdot y_{i-1}(x). \label{eq:DEs_min_degree}
\end{align}

It is easy to check that the assumptions of the DEs method are satisfied (we omit details since we did not formally introduce the tool). The conclusion is that a.a.s.\ during the entire Phase~$q$, for any $q \le i \le r$), $|Y_{i}(s) -y_{i}(s/n) n| = o(n)$. In particular, Phase~$q$ ends at time $s_q \sim x_q n$, where $x_q > x_{q-1}$ is the solution of the equation $y_q(x)=0$. Using the final values $y_i(x_q)$ in Phase~$q$ as initial values for Phase~$q+1$ we can repeat the argument inductively moving from phase to phase. 

We stop the analysis at the end of Phase $r$ when a graph of minimum degree equal to $r+1$ is reached. As discussed earlier, it happens at time $s > t$ and so we may ``rewind'' the process back to round $t$ to check the degree distribution of $G_t$. Based on that, we may compute $L(G_t)$ which gives us the desired lower bound for $\alpha(G_t)$. Suppose that round $t$ occurs during phase $q \le r$. A.a.s.\ there are $(1+o(1)) w(k) n$ vertices of degree $k \ge q$ in $G_t$, where
$$
w(k) =
\begin{cases}
y_k(t/n) & \text{ if } q \le k \le r \\
\frac {\lambda^k}{k!} \exp(-\lambda) & \text{ if } k \ge r+1.
\end{cases}
$$
These observations imply the following lower bound.

\begin{lemma}
Suppose that $t = t(n) \sim \lambda n$ for some $\lambda \in (0,\infty)$. Let $\epsilon > 0$ be any (arbitrarily small) constant. Then, there is no strategy that a.a.s.\ creates $G_t$ such that 
$$
\alpha(G_t) < (1-\epsilon) \, n \sum_{k \ge M} \frac {w(k)}{k+1}.
$$ 
\end{lemma}

\section{Conclusions}\label{sec:conclusions}

In this paper, we investigated three monotone properties. Our bounds are off by at most a multiplicative factor of $2+o(1)$. It would be interesting to close the gap between the upper and the lower bounds (or, at least, narrow them down).
\begin{itemize}
\item The property of containing a complete graph of order $k$ is well understood unless $t=t(n) = \Theta( n \log n)$. 
\item The property of creating a graph with the chromatic number at least $k$, is well understood when $t \gg n \log n$. More work is needed when $t = \bigo ( n \log n )$.
\item The property of not having an independent set of size at least $k$ remains to be investigated when $t = \Theta(n)$. In other regimes, the asymptotic behaviour is determined. 
\end{itemize}

\section{Acknowledgements}

This work was done while the authors were visiting the Simons Institute for the Theory of Computing.
The second author is supported in part by the Austrian Science Fund (FWF) [10.55776/I6502]. For the purpose of open access, the authors have applied a CC BY public copyright licence to any Author Accepted Manuscript version arising from this submission.

\bibliographystyle{plain}

\bibliography{refs.bib}

\begin{thebibliography}{10}

\bibitem{alon2016probabilistic}
Noga Alon and Joel~H Spencer.
\newblock {\em The probabilistic method}.
\newblock John Wiley \& Sons, 2016.

\bibitem{bayati2010combinatorial}
Mohsen Bayati, David Gamarnik, and Prasad Tetali.
\newblock Combinatorial approach to the interpolation method and scaling limits
  in sparse random graphs.
\newblock In {\em Proceedings of the forty-second ACM symposium on Theory of
  computing}, pages 105--114, 2010.

\bibitem{behague2022}
Natalie Behague, Trent Marbach, Pawe\l{} Pra\l{}at, and Andrzej Ruci\'nski.
\newblock Subgraph games in the semi-random graph process and its
  generalization to hypergraphs.
\newblock {\em arXiv preprint arXiv:2105.07034}, 2022.

\bibitem{beneliezer2020fast}
Omri Ben-Eliezer, Lior Gishboliner, Dan Hefetz, and Michael Krivelevich.
\newblock Very fast construction of bounded-degree spanning graphs via the
  semi-random graph process.
\newblock {\em Proceedings of the 31st Symposium on Discrete Algorithms
  (SODA'20)}, pages 728--737, 2020.

\bibitem{beneliezer2019semirandom}
Omri Ben-Eliezer, Dan Hefetz, Gal Kronenberg, Olaf Parczyk, Clara Shikhelman,
  and Miloš Stojaković.
\newblock Semi-random graph process.
\newblock {\em Random Structures \& Algorithms}, 56(3):648--675, 2020.

\bibitem{BD20}
Patrick Bennett and Andrzej Dudek.
\newblock A gentle introduction to the differential equation method and dynamic
  concentration.
\newblock {\em Discrete Mathematics}, 345(12):113071, 2022.

\bibitem{bohman2009hamilton}
Tom Bohman and Alan Frieze.
\newblock Hamilton cycles in 3-out.
\newblock {\em Random Structures \& Algorithms}, 35(4):393--417, 2009.

\bibitem{burova2022semi}
Sofiya Burova and Lyuben Lichev.
\newblock The semi-random tree process.
\newblock {\em arXiv preprint arXiv:2204.07376}, 2022.

\bibitem{caro1979new}
Yair Caro.
\newblock New results on the independence number.
\newblock Technical report, Technical Report, Tel-Aviv University, 1979.

\bibitem{dubhashi1998balls}
Devdatt Dubhashi and Desh Ranjan.
\newblock Balls and bins: A study in negative dependence.
\newblock {\em Random Structures \& Algorithms}, 13(5):99--124, 1998.

\bibitem{frieze2023building}
Alan Frieze, Pu~Gao, Calum MacRury, Pawe{\l} Pra{\l}at, and Gregory Sorkin.
\newblock Building hamiltonian cycles in the semi-random graph process in less
  than $2 n $ rounds.
\newblock {\em arXiv preprint arXiv:2311.05533}, 2023.

\bibitem{frieze2022hamilton}
Alan Frieze and Gregory~B Sorkin.
\newblock Hamilton cycles in a semi-random graph model.
\newblock {\em arXiv preprint arXiv:2208.00255}, 2022.

\bibitem{frieze1986maximum}
Alan~M Frieze.
\newblock Maximum matchings in a class of random graphs.
\newblock {\em Journal of Combinatorial Theory, Series B}, 40(2):196--212,
  1986.

\bibitem{gao2020hamilton}
Pu~Gao, Bogumił Kamiński, Calum MacRury, and Paweł Prałat.
\newblock Hamilton cycles in the semi-random graph process.
\newblock {\em European Journal of Combinatorics}, 99:103423, 2022.

\bibitem{gao2022fully}
Pu~Gao, Calum MacRury, and Pawel Pralat.
\newblock A fully adaptive strategy for hamiltonian cycles in the semi-random
  graph process.
\newblock In Amit Chakrabarti and Chaitanya Swamy, editors, {\em Approximation,
  Randomization, and Combinatorial Optimization. Algorithms and Techniques,
  {APPROX/RANDOM} 2022, September 19-21, 2022, University of Illinois,
  Urbana-Champaign, {USA} (Virtual Conference)}, volume 245 of {\em LIPIcs},
  pages 29:1--29:22. Schloss Dagstuhl - Leibniz-Zentrum f{\"{u}}r Informatik,
  2022.

\bibitem{gao2022perfect}
Pu~Gao, Calum MacRury, and Pawe{\l} Pra{\l}at.
\newblock Perfect matchings in the semirandom graph process.
\newblock {\em SIAM Journal on Discrete Mathematics}, 36(2):1274--1290, 2022.

\bibitem{gilboa2021semi}
Shoni Gilboa and Dan Hefetz.
\newblock Semi-random process without replacement.
\newblock In {\em Extended Abstracts EuroComb 2021}, pages 129--135. Springer,
  2021.

\bibitem{JLR}
Svante Janson, Tomasz \L{}uczak, and Andrzej Ruci\'nski.
\newblock {\em Random graphs}, volume~45.
\newblock John Wiley \& Sons, 2011.

\bibitem{Karonski_Frieze}
Michal Karo\'nski and Alan Frieze.
\newblock {\em Introduction to Random Graphs}.
\newblock Cambridge University Press, 2016.

\bibitem{pittel}
Michal Karoński, Ed~Overman, and Boris Pittel.
\newblock On a perfect matching in a random digraph with average out-degree
  below two.
\newblock {\em Journal of Combinatorial Theory, Series B}, 143, 03 2020.

\bibitem{koerts2022k}
Hidde Koerts.
\newblock k-connectedness and k-factors in the semi-random graph process.
\newblock Master's thesis, University of Waterloo, 2022.

\bibitem{Kurtz1970}
Thomas~G. Kurtz.
\newblock Solutions of ordinary differential equations as limits of pure jump
  {M}arkov processes.
\newblock {\em J. Appl. Probability}, 7:49--58, 1970.

\bibitem{macrury2022sharp}
Calum MacRury and Erlang Surya.
\newblock Sharp thresholds in adaptive random graph processes.
\newblock {\em Random Structures \& Algorithms}, 64(3):741--767, 2024.

\bibitem{molloy2023matchings}
Michael Molloy, Pawel Pralat, and Gregory~B Sorkin.
\newblock Matchings and loose cycles in the semirandom hypergraph model.
\newblock {\em arXiv preprint arXiv:2401.00559}, 2023.

\bibitem{Harjas}
Pawe{\l} Pra{\l}at and Harjas Singh.
\newblock Power of $k$ choices in the semi-random graph process.
\newblock {\em Electronic Journal of Combinatorics}, 31(1):\#P1.11, 2024.

\bibitem{Warnke2020}
Lutz Warnke.
\newblock On {W}ormald's differential equation method.
\newblock {\em Combin. Probab. Comput.}, in press.

\bibitem{wei1981lower}
VK~Wei.
\newblock A lower bound on the stability number of a simple graph.
\newblock Technical report, Bell Laboratories Technical Memorandum New Jersey,
  1981.

\bibitem{W1995}
Nicholas~C. Wormald.
\newblock Differential equations for random processes and random graphs.
\newblock {\em Ann. Appl. Probab.}, 5(4):1217--1235, 1995.

\bibitem{W1999}
Nicholas~C. Wormald.
\newblock The differential equation method for random graph processes and
  greedy algorithms.
\newblock In {\em {In Lectures on Approximation and Randomized Algorithms, PWN,
  Warsaw}}, pages 73--155, 1999.

\end{thebibliography}

\end{document}